\documentclass[twoside,a4paper,10pt]{article}
\usepackage{amsfonts, amsbsy, amsmath, amsthm, amssymb, latexsym, verbatim, enumerate}
\usepackage{mathrsfs}
\usepackage[top=34mm,right=34mm,bottom=34mm,left=34mm]{geometry}
\usepackage{stmaryrd}
\usepackage{bm}

\newcommand{\GL}{{\operatorname{GL}}}

\newcommand{\SL}{{\operatorname{SL}}}
\newcommand{\PSL}{{\operatorname{PSL}}}

\newcommand{\SU}{{\operatorname{SU}}}

\newcommand{\U}{{\mathbf U}}
\newcommand{\Lin}{{\mathbf L}}
\newcommand{\Sp}{{\mathbf S}}
\newcommand{\Cl}{{\it Cl}}

\newcommand{\Stab}{{\rm Stab}}
\newcommand{\Isom}{{\rm Isom}}

\renewcommand{\a}{\alpha}

\renewcommand{\l}{\lambda} 
\renewcommand{\O}{{ {\mathbf O}}}

\newcommand{\Om}{{\mathbf\Omega}}

\newtheorem{theorem}{Theorem} 

\newtheorem{thm}{Theorem}[section] 
\newtheorem{prop}[thm]{Proposition} 
\newtheorem{lemma}[thm]{Lemma}
\newtheorem{cor}[thm]{Corollary} 
 
\newtheorem{rmk}[thm]{Remark}

\theoremstyle{definition}
\newtheorem{defn}[thm]{Definition}

\newtheorem*{rem}{Remark}

\newtheorem{cla}[thm]{Claim}

\begin{document}
\title{On the irreducibility of symmetrizations of cross-characteristic
representations of finite classical groups}

\author{Kay Magaard\footnote{\texttt{k.magaard@bham.ac.uk}} \\ University of Birmingham\\ School of Mathematics, Watson Building\\
Edgbaston, Birmingham, B15 2TT, UK  \and
Gerhard R\"ohrle\footnote{\texttt{gerhard.roehrle@rub.de}} \\Ruhr-Universit\"at Bochum\\
Fakult\"at f\"ur Mathematik,
Universit\"atstrasse 150\\Bochum D-44780 \and
Donna M. Testerman\footnote{\texttt{donna.testerman@epfl.ch}} \\{E}cole Polytechnique F\'{e}d\'{e}rale de Lausanne \\ Section de Math\'ematiques, MATHGEOM, Station 8 \\ CH-1015 Lausanne, Switzerland}

\date{\today}
\maketitle

\begin{abstract}
Let $W$ be a vector space over an algebraically closed field $k$. Let $H$ be a quasisimple group of Lie type of 
characteristic $p\ne {\rm char}(k)$ acting  irreducibly on $W$. Suppose also that $G$ is a classical group with natural
module $W$, chosen minimally with respect to containing the image of $H$ under the associated representation. 
We consider the question of when $H$ can act irreducibly on a $G$-constituent of $W^{\otimes e}$ and study its
relationship to the maximal subgroup problem for finite classical groups.
\end{abstract}

\begin{section}{Introduction}\label{sec:intro}

Let $C(V)$ be a finite classical group with natural module $V$.
In his celebrated paper (\cite{Asch84}), 
Aschbacher defined a collection $\mathcal{C}(C(V))$ of subgroups of $C(V)$, and proved that maximal subgroups
of $C(V)$ which are not members of $\mathcal{C}(C(V))$ are necessarily of the form 
$N_{C(V)}(H)$, where $H$ is a quasisimple group acting absolutely irreducibly on $V$. 
Conversely, it is far from clear whether an absolutely irreducible $H$-module $V$ leads to a maximal 
subgroup of $C(V)$. For example, a possible obstruction to the maximality of $N_{C(V)}(H)$ 
could be the existence of 
a quasisimple group $G$ containing $H$ such that $N_{C(V)}(H) \subset N_{C(V)}(G) $. In this case 
we see that $V$ is also an absolutely irreducible $G$-module whose restriction to $H$ is irreducible. 
Thus we are naturally lead to consider branching problems in the representation theory
of quasisimple groups. Our main theorem is the solution of one particular branching problem,
whose corollary is a contribution to the determination of the maximal subgroups of the finite classical groups.

Throughout we consider the following situation. Let $H$ be a perfect finite group of Lie type of characteristic 
$p$, and $W$ a 
vector space over an algebraically closed field
$k$ of characteristic distinct from $p$. If $\rho: H \rightarrow \GL(W)$ is an
irreducible representation of $H$, then the assumption 
that $H$ is perfect implies that $H\rho \subset \SL(W)$. The Frobenius-Schur 
indicator of $\rho$ determines whether
or not $H\rho$ stabilizes a nondegenerate bilinear or quadratic form on 
$W$. If the indicator is zero, define 
$G: = \SL(W)$, else take the
connected component $C$ of the stabilizer of the nondegenerate bilinear or 
quadratic form 
stabilized 
by $H\rho$ and define $G$ to be the simply connected version of the simple 
algebraic group 
determined by 
$C$. Let $V$ be a finite-dimensional 
irreducible $kG$-module. We consider the question of when the 
restriction of $V$ to $H\rho$ is reducible; 
it is reasonable to expect that in most cases this is indeed the case. 
Our main theorem shows 
that this is so if we assume that $H$ is classical and $W$ satisfies a 
certain technical 
hypothesis which we call \emph{$Q$-linear large}. (The precise definitions of 
\emph{classical} and 
\emph{$Q$-linear large} are given below.)

\begin{theorem} Let $p$ be a prime, $k$ an algebraically closed field of 
characteristic $\neq p$.
Let $G$ be a simply connected simple algebraic group of classical
type with natural module 
$W = k^m$. Let $H$ be a 
quasisimple finite classical group of type $X_r(q)$, $q = p^f$, 
of untwisted Lie rank $r$ and characteristic $p$.
Suppose that $W$ is a $Q$-linear large irreducible $kH$-module. If 
$H \subset G $, then for any 
restricted irreducible $kG$-module $V$ not isomorphic to $W$ or $W^*$,
we have that 
$V|_H$ is reducible.
If $V$ is a twisted tensor product of restricted irreducible $kG$-modules, 
then 
$V|_H$ is reducible unless $q \leq 3$, 
$H$ is not a central extension of ${\rm PSL}_n(q)$, and 
$V$ is a Frobenius twist of a module of the form $X \otimes Y^\delta$, where 
$X,Y\in\{W,W^*\}$, and $Y^\delta$ is a Frobenius twist of $Y$
such that $X|_H$ and $Y^\delta|_H$ are inequivalent $kH$-modules. 
\end{theorem}


Our theorem has an immediate corollary.

\begin{cor} Let $p$ be a prime, $k$ an algebraically closed field of 
characteristic $\neq p$.
Let $G$ be a classical group with natural module $V$ defined over $k$, and
let $H$ be a quasisimple finite classical group of Lie type $X_r(q)$, $q = p^f$, of 
untwisted Lie rank $r$ and characteristic $p$.
Assume that  $F$ is a Frobenius morphism of $G$ such that $H \subset G^{F}$ and $V^{F}$ is the 
natural module for $G^{F}$. If
$H$ acts absolutely irreducibly on $V^{F}$ and there exists a
classical simple algebraic group $G'$  of 
characteristic ${\rm char}(k)$ and a Frobenius morphism $F'$ of $G'$ such that 
$H \subset (G')^{F'} \subset G^F$, then either the natural module $W$ of $G'$ is not 
$Q$-linear large for $H$, or 
$q \leq 3$, $V$ is a Frobenius twist of a module of the form $X\otimes Y^\delta$, where 
$X,Y \in \{ W, W^*\}$, and $Y^\delta$ is a Frobenius twist of $Y$
such that $X|_H$ and $Y^\delta|_H$ are inequivalent $kH$-modules.
\end{cor}

Our corollary sheds more light on the set of maximal subgroups of 
finite classical groups. To see this suppose that 
we would like to study the maximal subgroups of a 
finite classical group $G^F$ whose natural module is $V^F$ as above.
Suppose also that $H \subset G^F$ is of type $X_r(q)$, acts absolutely irreducibly
on $V^F$, and that $(q,{\rm char}(k))=1$. 
What are the possible overgroups of $H$ in $G^F$? In the language of 
Aschbacher's reduction theorem (\cite{Asch84}), a possible overgroup $K$ might 
be a member of one of the following families $\mathcal{C}_2$, $\mathcal{C}_4$, 
$\mathcal{C}_6$, 
$\mathcal{C}_7$, or $\mathcal{S}$.

The situation where $K$ is a member of $\mathcal{C}_2$, that is where $H$ acts 
imprimitively on $V^F$,
is being investigated in \cite{HHM}. The situation where $K$ is a member 
of 
$\mathcal{C}_4$
was investigated by Magaard and Tiep \cite{MaTi}. The situation where $K$ is a
 member of 
$\mathcal{C}_6$ or 
$\mathcal{C}_7$ is being investigated in \cite{MaTi12} and independently  
in \cite{Bray}. 

The situation where $K$ is a member of $\mathcal{S}$ naturally 
subdivides according to the type of the simple group $K$. Evidently $V^F$ is 
an irreducible 
$K$-module
whose restriction to $H$ is absolutely irreducible. Thus determining the 
maximal subgroups 
of finite
classical groups
requires the solution of certain branching rule problems. 
Branching rules for alternating and symmetric groups are an active area of 
current research. See for example \cite{JS,KT, BK, BO, KS}.
The case where $K$ is a finite group of Lie type
whose defining characteristic is equal to the defining characteristic of $H$
 was 
investigated by Seitz
in \cite{SeitzCC}. He shows that if $q > 3$ then the possible pairs $(H,K)$ 
form 
four families of examples.

In this paper we discuss the case where $K = (G')^{F'}$ is a member of 
$\mathcal{S}$, and is 
a finite classical group of 
Lie type over a field of 
characteristic equal to ${\rm char}(k)$. Work of Malle \cite{M}, 
Magaard and Tiep \cite{MaTi}, and 
Magaard, Malle and Tiep \cite{MMT} shows that there are families of examples. 
One such is 
obtained when $H = {\rm Sp}_{2r}(3)$ and $V$ is the alternating or symmetric 
square of a 
Weil module $W$ of $H$.
In this case $G = \SL(V)$ contains an irreducible subgroup $H$ with
 intermediate subgroup 
$\SL(W)$.
Our main theorem shows that in quantifiable terms, such configurations are 
rare. To make 
this explicit,
we now give the precise definition of \emph{$Q$-linear large}.

Let $H$ be a finite classical group with natural module $N$ and let 
$P<H$ be the 
stabilizer of a 
singular $1$-space of $N$. Then $P=QL$, where $Q=O_p(P)$. We write 
$Q^*$ for the 
space ${\rm Hom}(Q/[Q,Q],k^*)$. For $\chi\in Q^*$ and a $kH$-module $W$, define  
$W_\chi: =\{w\in W\mid xw=\chi(x)w\mbox{ for all } x\in Q\}$.

\begin{defn}\label{lin-spe} Let $P$ be as above. We say we are in the 
{\em $Q$-linear case} (or that \emph{$W$ is a $Q$-linear module for $H$}) 
if either $Q$ is abelian or $Q$ is non-abelian and 
$[W,Q]\ne[W,[Q,Q]]$. We say we are
in the \emph {$Q$-special} case otherwise.
\end{defn}

\begin{rem}  Note that we are in the $Q$-linear case when
$H$ is linear or orthogonal. In addition, if we are in the $Q$-linear case, 
then 
 there exists
$\chi\in Q^*$, $\chi\ne 1$, with $W_\chi\ne 0$.
\end{rem}

We now make a further case distinction, splitting the family of $Q$-linear 
modules
into two subfamilies.
For $\chi\in Q^*$, set $P_\chi:=\Stab_P(W_\chi)$; then
$P_\chi=QL_\chi$, where $L_\chi = \Stab_L(W_\chi)$. 
Write $L_\chi^\infty$ for the last term in the derived series of $L_\chi$. 

\begin{defn}\label{linlarge} Let $P$ be as above and assume we are in the 
$Q$-linear case.
\begin{enumerate}[(a)]
\item We say that we are in the {\em $Q$-linear large} case, or that \emph{$W$ is a 
$Q$-linear large module for $H$}, if there exists
$\chi\in Q^*$, $\chi\ne 1$, such that
${\rm soc}(W_\chi|_{L_\chi^\infty})$ has an irreducible summand $S$ of dimension 
greater than $1$.
\item We say that we are in the {\em $Q$-linear small} case, or that 
\emph{$W$ is a $Q$-linear small module for $H$}, if 
for all $\chi\in Q^*$, $\chi\ne 1$, 
all irreducible constituents of ${\rm soc}(W_\chi|_{L_\chi^\infty})$ are 
$1$-dimensional.

\end{enumerate}
\end{defn}

Now returning to the example mentioned above, where $H = {\rm Sp}_{2r}(3)$ and 
$W$ is a 
Weil module for $H$,
we see that the statement of our main theorem fails if we drop the $Q$-linear large
hypothesis. Indeed, the Weil modules are $Q$-special and hence not 
$Q$-linear large. 

For finite classical groups $H$ of large untwisted Lie rank $r$ acting on an ${\mathbb F}_{q}$-vector space 
(or ${\mathbb F}_{q^2}$-space if $H$ is a unitary group), the $Q$-linear large 
hypothesis is not very restrictive in the 
sense that the proportion of $Q$-linear large modules tends to $1$ as $q$ and $r$ tend to infinity.
To see this, note that the dimension of a $Q$-linear small module is bounded 
above by 
$[H:QL_\chi^\infty]$ and the results of Guralnick, Magaard, Saxl, and Tiep in 
\cite{GMST} show that 
the 
$Q$-special hypothesis implies that 
if $H$ is unitary, then $W$ is one of at most $(q+1)$ Weil modules, and 
that if $H$ is 
symplectic, then $W$ is either 
one of four Weil modules 
or of dimension bounded above $\frac{(q^r-1)(q^{r-1}-q)}{2(q+1)}$. In the latter case it follows by 
inspection that the module dimension is bounded by 
$[H:L_\chi^\infty]$.  So in summary, if an $H$-module $W$ is not $Q$-linear large, then its dimension is bounded by 
$[H:QL_\chi^\infty]$. 

Now, according to Deligne-Lusztig theory, the (ordinary) Deligne-Lusztig character $\pm R_{T,\theta}$ is 
irreducible when $\theta$ is in general position. Moreover $|R_{T,\theta}(1)| = [H:T]_{p'}$ is a polynomial 
function in $q$ of degree $|\Phi^+|$, where $\Phi^+$ denotes the set of positive roots of the root system used 
to define $H$. For $H$ classical $|\Phi^+|$ is a quadratic function in $r$ which shows that a module affording an irreducible
$\pm R_{T,\theta}$ is $Q$-linear large, provided that $r \geq 4$.

Next we observe that each irreducible $\pm R_{T,\theta}$ corresponds uniquely to a  
regular semisimple class in $H^*$, the Deligne-Lusztig dual of $H$. When $r$ and $q$ tend to infinity
then the proportion of regular semisimple classes of $H^*$ to the total number of conjugacy classes of $H$
tends to $1$. 

Now if $\ell = {\rm char}(k)$ and $\ell$ is a divisor of $|H|$, then the reduction of $\pm R_{T,\theta}$
modulo $\ell$ stays irreducible if the corresponding regular semisimple class in $H^*$ is represented by 
an element of order coprime to $\ell$. Again when $r$ and $q$ tend to infinity,
then the proportion of regular semisimple $\ell'$-classes of $H^*$ to the total number of $\ell'$
conjugacy classes of $H$ tends to $1$, and thus the proportion of $Q$-linear large characters of 
$H$ tends to $1$.

Our results are closely related to conjectures of Larsen and of Koll\'ar and 
Larsen, both of 
which were proved by 
Guralnick and Tiep in \cite{GT1} respectively \cite{GT2}. For a group $X$, an 
$X$ module 
$W =k^m$, and a 
positive integer $j$, define 
$M_{2j}(X,W)$ to be $\dim({\rm End}_X(W^{\otimes j}))$. If $G$ is a classical group 
with natural module $W$, and $H$ is a 
finite subgroup of $G$, then Larsen conjectured that if ${\rm char}(k) = 0$, then $M_{2j}(H,W) > M_{2j}(G,W)$ for 
some 
$j \leq 4$.
Note that if ${\rm char}(k) = 0$, then $M_{2j}(H,W) = M_{2j}(G,W)$ implies 
that every irreducible 
$G$-constituent of 
$W^{\otimes j}$ restricts irreducibly to $H$. 
 The conjecture of Koll\'ar and Larsen, now a theorem \cite{GT2}, 
concerns the action of 
$H$ on ${\rm Sym}^j(W)$ and it roughly asserts that most of the time 
$H$ cannot act irreducibly for $j \geq 4$.

For the class of cross-characteristic $Q$-linear large modules of classical 
groups our 
results yield stronger versions of both conjectures. 

\begin{theorem}
Let $H = X_{r}(q)$ be a quasisimple finite classical group, 
cross-characteristically embedded in a simple algebraic group $G$ of classical
type, whose natural module is 
$W$. If 
$W$ is a $Q$-linear large $H$-module, then we have 
that 
no $kG$-constituent of $W^{\otimes a} \otimes (W^*)^{\otimes b}$ restricts 
irreducibly to 
$H$ for any $a+b=j\geq 2$. In particular, if ${\rm char}(k) = 0$, then
 $M_{2j}(H,W) > M_{2j}(G,W)$ for all $j \geq 2$.
\end{theorem}

We now make some remarks about how we prove our result. Assume that 
$W$ is a 
$Q$-linear large 
$kH$-module and that $V$ is a restricted irreducible module for $G$, the 
classical
simple algebraic group with natural module $W$.  Let 
$\lambda = \sum_{i=1}^\ell a_i\lambda_i$ be the highest
weight of $V$, where $\lambda_i$ are the fundamental weights with respect to a 
fixed
maximal torus and Borel subgroup of $G$ and $\ell = {\rm rank}(G)$. If 
$G \neq {\rm SL}(W)$
set $e(\lambda):=\sum_{i=1}^\ell a_ii$; for $G = {\rm SL}(W)$
set  $e(\lambda):=\sum_{i=1}^{\ell} {\rm min}(i,\ell+1-i)a_i$. 
Let $B$ denote an upper bound for the dimension of an irreducible $kH$-module,
 as given in 
\cite[Thm 2.1]{SeitzCC}. Using a result
of Premet (\cite{Premet88}), we are guaranteed that the weight spaces for certain 
weights 
subdominant to 
$\lambda$ are nontrivial.
We show first that if $e(\lambda) \geq (\ell+1)/2$, then there exists 
a subdominant weight $\mu$, occurring with nonzero multiplicity in $V$, such that 
the length of the Weyl group orbit of 
$\mu$ is greater than $B$. Thus from this point on we may assume that 
$e(\lambda) < (\ell+1)/2$.
Now for restricted modules $V$ for which $e(\lambda) < (\ell+1)/2$, we produce an 
$L_\chi$-invariant submodule
$V_0$, which by Frobenius reciprocity yields an upper bound of $[H:L_\chi]\dim V_0$ for 
$\dim V$.
On the other hand, the $Q$-linear large hypothesis yields a lower bound for $\dim V$ of
$(\dim V_0) 2^{e(\lambda)}{\ell \choose e(\lambda)}$.
Comparing these bounds for $\dim V$ leads to a contradiction 
if $e(\lambda) \geq 3$. When $e(\lambda) = 2$ we use \cite{MMT} to conclude that none of 
the examples 
which occur there satisfy the $Q$-linear large hypothesis, proving the first part of our 
theorem.
Finally, we use our result for restricted 
modules to deduce that if $V$ is a twisted tensor product, then all factors 
must be Frobenius 
twists of $W$ or its dual. Then we repeat the argument with the upper and 
lower bounds for 
$\dim V$ to conclude that the number of factors is at most $2$. Then we use
the main theorem of \cite{MaTi} to draw our conclusion.
 
It is natural to ask how difficult it is to remove the hypothesis that
$W$ is
$Q$-linear large.
Firstly, there are examples of triples $(H,W,V)$,
with $W$ a $Q$-linear small $kH$-module, and $V$ a $kG$-module on which 
$H$ acts irreducibly, both with $V$ restricted and with $V$ non-restricted.
For example, 
if $H = \Omega_7(2) \cong {\rm Sp}_6(2)$ and $W$ is seven-dimensional and 
irreducible, 
then $W$ is not $Q$-linear large and the restricted $H$-module 
$V = \Lambda^3(W)$ is irreducible. Another example arises when 
$H = 4_1.{\rm PSL}_3(4)$. (We are using the notation of \cite{Atlas}.) Here $H$ 
possesses exactly four inequivalent ordinary
irreducible characters of degree 8, likewise also for all characteristics
 $s$ greater than
$7$. If $s \equiv 2 \ \mbox{mod} \ 5$, then the standard Frobenius morphism 
$F$ induces a permutation of order two
on the characters of degree 8. So if $W$ is a module affording one of 
the characters of degree
8, then $W \otimes W^F$ is an irreducible $H$-module. 
There are even infinite series of examples; if $H = {\rm Sp}_{2r}(3)$, 
with $r \geq 2$ and $\dim W = \frac{q^r \pm 1}{2}$, then 
both $\Lambda^2(W)$ and ${\rm Sym}^2(W)$ are irreducible. We refer the reader 
to \cite{MMT, MaTi} for further examples
of irreducible alternating, symmetric  and tensor squares of non-$Q$-linear 
large modules. So we see that the $Q$-linear large hypothesis is necessary for
some choices of $H$ and $W$.

Secondly, we remark that the initial parts of the analysis used
here can
 be applied to the general case, i.e., to the $Q$-linear small case and to the
 non $Q$-linear case.
However, establishing lower and upper bounds for $\dim V$ in terms of
$e(\lambda)$ requires a detailed knowledge of the module $W$. Our expectation 
is that 
for large values of $q$ our results generalize 
completely, while for small values of $q$ we expect to obtain slightly
weaker results. Note that one of the examples above shows that the 
irreducible module $V$ may occur in the $3$-fold tensor power of $W$.
Nevertheless we expect that for 
every $q$ there exists an integer $e(q)$ such that for all restricted weights
$\lambda$ with $e(\lambda) > e(q)$, 
the irreducible $kG$-module with highest weight $\lambda$ is a reducible
$kH$-module. Our initial investigations
into the non $Q$-linear large case lead us to believe that for $q$ large
enough, $e(q)\leq 4$. The worst case encountered thus far is the case
where $W$ is a Weil module
for $H = {\rm Sp}_{2r}(3)$, and our methods give $e(q) \leq 17$.  The
analysis of the non $Q$-linear large
case is the subject of a forthcoming paper.

Finally we gratefully acknowledge a set of handwritten notes that was generously provided to 
us 
by Martin Liebeck and Gary Seitz, in which they outlined a strategy for proving our result.
In addition, we thank the referee for a very careful reading of the paper and for suggesting various
improvements and pointing out some essential corrections.

\end{section}


\begin{section}{Notation}\label{sec:notation}

Fix a vector space $N$ of dimension $n$ 
over a finite field $K$ of characteristic $p$.
Endow $N$ either with the trivial bilinear form,
a nondegenerate alternating bilinear form, a nondegenerate quadratic form
of Witt defect $0$ or $1$, or a nondegenerate unitary form, respectively.
If the form on $N$ is unitary, we assume that $K={\mathbb F}_{q^2}$, 
and otherwise we take $K={\mathbb F}_{q}$, where $q=p^a$, for $a\in{\mathbb N}$. 
Let ${\rm Isom}(N)$ denote the full group of isometries of $N$. We assume that $F^*({\rm Isom}(N))$ is
a quasisimple group. We recall that $F^*(X)$ denotes the generalized Fitting subgroup of a finite group $X$.

\begin{defn} A quasisimple group $X$ is said to be a 
\emph {finite classical group of type} 
$\Lin_n(q)$, $\Sp_{2n}(q)$, ${\mathbf O}_n(q)$, 
respectively $\U_n(q)$, provided $X$ is a central extension
of the factor group $F^*({\rm Isom}(N))/Z(F^*({\rm Isom}(N)))$, where $N$ is
endowed as above with a trivial, alternating, quadratic, 
respectively unitary form.
\end{defn}

By abuse of notation, we write $H=\Lin_n(q)$, $\Sp_{2n}(q)$, etc
to mean that $H$ is of type $\Lin_n(q)$, $\Sp_{2n}(q)$, etc. For 
example, when we write $H=\Sp_4(2)$,
we mean that $H$ is either the alternating group on $6$ letters or one
of its cyclic central extensions. We 
 use the notation $\Om^+_{2m}(q)$ and $\Om^-_{2m}(q)$
in place of ${\mathbf O}_{2m}(q)$, when we need to 
distinguish between the two types of even-dimensional
orthogonal groups. The $+$ type groups correspond to the case where the
quadratic form has Witt index $m$ and the $-$ type ones
to the case where the quadratic form has Witt index $m-1$.

\begin{rem}

\begin{enumerate}

\item Since $\Sp_{2n}(2^a)\simeq ({\mathbf O}_{2n+1}(2^a))'$,
in the context of this paper, it is
more convenient to
 regard symplectic groups over fields of characteristic $2$
as orthogonal groups. Nevertheless, we allow ourselves to pass back and forth between
these two points of view, since  certain of the results taken from other
sources view these groups as symplectic groups.

\item In addition, since certain
of the groups in small dimension are isomorphic, we assume further that 
$n\geq2$ for type $\Sp_{2n}(q)$, $n\geq3$ for type $\U_n(q)$,
$n\geq5$ for type ${\mathbf O}_n(q)$, $n$ odd, and 
$n\geq8$ for type ${\mathbf O}_n(q)$, $n$ even.

\item Finally, we note that we write $\PSL_n(q)$, ${\rm U}_n(q)$, 
${\rm Sp}_n(q)$,
$\Omega_n^\pm(q)$ and ${\rm O}_n(q)$,
when we want to emphasize that we are not dealing with a central
extension, but rather with a subgroup of the linear isometry group.
\end{enumerate}
\end{rem}

\begin{defn} Each finite classical group has a \emph{rank}, denoted $r$,
defined in Table~\ref{tab:rank}. In order to avoid duplication, we
make various assumptions on $r$ as indicated.

\begin{table}[!h]
$$\begin{array}{|l|l|l|} \hline
H&r&\mbox{condition}\\
\hline
\Lin_n(q)& n-1& r\geq 1\\
\U_n(q)&n-1&r\geq2\\
\Sp_{2n}(q)&n&r\geq 2\\ 
{\mathbf O}_{2n}(q)&n&r\geq 4\\
{\mathbf O}_{2n+1}(q)&n&r\geq2, q \mbox{ even}\\
&&r\geq3, q \mbox{ odd}\\
\hline
\end{array}$$
\caption{Rank $r$ of finite classical groups}
\label{tab:rank}
\end{table}

\end{defn}

We rely upon many basic structural results about quasisimple finite groups
of Lie type, for example their orders, the index of a Borel subgroup,
exceptional isomorphisms, as well as results on the sizes of orbits of vectors
in the classical geometries. The basic resources for these results
are \cite{Carter_little} and \cite{Grove}.

For a finite quasisimple group $K$, we write $K^\infty$ for the last term in the derived series of $K$.
\end{section}

\begin{section}{Some combinatorial lemmas}

We begin with a combinatorial lemma; the proof of (2) was kindly provided by Karen Collins.

\begin{lemma}\label{lem:bc-app}
Let $j$, $m$, $a$ and $b$ be natural numbers such that $1\leq j\leq m^{1/2}$, $a,b<\frac{m+1}{2}$ and 
$\frac{m}{2}<a+b\leq m$. 
Set $d = \lfloor\frac{m}{2}\rfloor$. The following inequalities hold.
\begin{enumerate}[]
\item{\rm (1)} ${m\choose j}> m^{j/2}$;
\item{\rm (2)} $\frac{m!}{a!b!(m-a-b)!}\geq{m\choose d}$.
\end{enumerate}
\end{lemma}


\begin{proof}

For $j\leq m$, we note that $\frac{m-i}{j-i}\geq\frac{m}{j}$
 for all $i\leq j-1$. Thus ${m\choose j}\geq(\frac{m}{j})^j$, with strict
inequality if $j<m$. 
The inequality of (1) follows directly.

For (2), first suppose $a+b=m$; since $a,b\leq\frac{m}{2}$, we have 
$m$ even, $d=m/2$ and $a=\frac{m}{2}=b$. But then the inequality is clear.

So now assume $a+b<m$. First consider the case where $a\geq d$. Since  
$a+b<m$, $b< m-a$, so ${{m-a}\choose {b}}> 1$.
Since $a, d \leq m/2$ and $a \geq d$, then 
${m\choose a}\geq {m\choose d}$.
So $m!/(a!b!(m-a-b)!) = {m\choose a}{m-a\choose b} > {m\choose d}$ as claimed.

So we may now assume $a<d$, and by symmetry, $b<d$ as well. Recall
 that $d\leq m/2<a+b$. We then have 

$$\frac{a!b!(m-a-b)!}{d!(m-d)!}=
\frac{a!b!(m-a-b)!(d-b)!(a+b-d)!}{d!(m-d)!)(d-b)!(a+b-d)!},$$

which is equal to

$$\frac{a!}{(d-b)!(a+b-d)!}\cdot\frac{b!(d-b)!}{d!}\cdot
\frac{(m-a-b)!(a+b-d)!}{(m-d)!}.$$

This latter equals ${a\choose d-b}/({d\choose {d-b}} {{m-d}\choose{a+b-d}})$.

Now $d > a$, so ${d\choose{d-b}} > {a\choose{d-b}}$ and 
${{m-d}\choose{ a+b-d}}>1$, since 
$m > a+b$.
Hence the numerator is less than the denominator, and therefore, 
$a!b!(m-a-b)!$ is less than $d!(m-d)!$, as desired.\end{proof}

\begin{lemma}\label{bc-lower} For $\ell \geq 2$, ${\ell\choose{\lfloor\ell/2\rfloor}} > 2^{\ell}/(\ell+1)$.

\end{lemma}

\begin{proof} It suffices to note that $2^\ell$ is the sum of $\ell+1$ binomial coefficients, the largest of which is ${\ell\choose{\lfloor\ell/2\rfloor}}$.\end{proof}

\end{section}

\begin{section}{Some weight theory}

Let $G$ denote a simple algebraic group of classical type 
defined over an algebraically
closed field $k$ with ${\rm char}(k)\ne p$. That is, we fix a finite-dimensional
vector space $W$ defined over $k$ and equip $W$ with 
either the trivial form or a nondegenerate alternating bilinear
form, or a nondegenerate quadratic form and take $G$ to be 
$(({\rm Isom}(W))')^\circ$. By nondegenerate, we mean that the radical
of the underlying bilinear form is trivial, with one
exception: If ${\rm char}(k)=2$ and $W$ is an odd-dimensional orthogonal
space, we say $W$ is ``nondegenerate'' provided the radical of the underlying 
bilinear form is a non-singular $1$-space. By a slight (conventional) abuse
of notation, 
we write $G={\rm SL}(W)$, ${\rm Sp}(W)$, respectively ${\rm SO}(W)$. When
we wish to cover all cases at once, we simply refer
to $G$ as $\Cl(W)$.

Let $T_G$ be a maximal torus of $G$, with character group $X(T_G)$, 
let $\Phi(G)$ denote the root system of 
$G$ relative to $T_G$, and take 
$\Pi(G)=\{\alpha_1,\alpha_2,\dots,\alpha_\ell\}$ 
to be a fundamental system of $\Phi(G)$, and therefore ${\rm rank}(G)=\ell$. 
We label Dynkin diagrams as in 
\cite{Bourb4-6}. 
Let $\lambda_i\in X(T_G)$ denote the fundamental dominant weight
corresponding to $\alpha_i$. Set $W_G=N_G(T_G)/T_G$, the Weyl group of $G$. 
For a dominant weight $\lambda\in X(T_G)$,
we denote the rational irreducible $kG$-module with highest weight $\lambda$ by 
$V_G(\lambda)$. (All $kG$-modules considered are rational and hence we do
not mention this in what follows.) Recall that a dominant weight
is said to be \emph{restricted} (and the corresponding irreducible
module is referred to as \emph{restricted} as well), if $\lambda = \sum_{i=1}^\ell a_i\lambda_i$ with
$a_i<{\rm char}(k)$ for all $i$.

We require the following definition.

\begin{defn}\label{e(lambda)} For a restricted dominant weight 
$\mu=\sum_{i=1}^\ell a_i\lambda_i\in X(T_G)$, set
\begin{equation*}
e(\mu)=\left\{\begin{array}{ll}
\sum_{i=1}^\ell ia_i,&
\mbox{if $G$ is of type $B_\ell, C_\ell$ or 
$D_\ell$,} \\
\sum_{i=1}^{\ell} {\rm min}(i,\ell+1-i)a_i, & \mbox{if $G$ 
is of type $A_\ell$.} \\
\end{array}\right.
\end{equation*}
\end{defn}
In addition, we set

\begin{equation*}
l_\mu= \left\{\begin{array}{ll}
0,& \mbox{if $a_c=0$ for all $c\leq\frac{\ell+1}{2}$;}\\
\mbox{max $\{c\ |\  1\leq c\leq\frac{\ell+1}{2}$ and $a_c\ne0\}$}, &
 \mbox{otherwise.}\\
\end{array}\right.
\end{equation*}

\begin{equation*}
r_\mu= \left\{\begin{array}{ll}
0,& \mbox{if $a_c=0$ for all $c>\frac{\ell+1}{2}$;}\\
\mbox{max $\{c\ |\  1\leq c<\frac{\ell+1}{2}$ and $a_{\ell+1-c}\ne0\}$}, & 
\mbox{otherwise.}\\
\end{array}\right.
\end{equation*}

We now prove a sequence of lemmas, whose aim is to establish a lower bound for 
$\dim V_G(\lambda)$
in terms of the rank $\ell$ of $G$ and $e(\lambda)$. The first result is taken from \cite{Seitzclass}.

\begin{lemma}\label{Weylgrp-stab}{\rm \cite[1.10]{Seitzclass}} Let $\mu$ 
be a dominant weight for 
$T_G$ and $W_0$ the 
subgroup
of $W_G$ generated by those fundamental reflections corresponding to 
fundamental roots $\alpha\in\Pi(G)$ with $\langle\mu,\alpha\rangle=0$. Then 
$W_0$ is the full stabilizer of $\mu$ in $W_G$. In particular,
the orbit of $\mu$ under the action of $W_G$ contains 
$|W_G:W_0|$ distinct weights.
\end{lemma}

\begin{lemma}\label{lem:old1} Let $\mu\in X(T_G)$ be a restricted 
dominant weight, 
$\mu=\sum_{i=1}^\ell a_i\lambda_i$ and set $d_0={\rm max}\{d\ |\  a_d\ne0\}$. 
Assume $a_{d_0}>1$ and in addition assume $d_0<\frac{\ell}{2}$ for $G$ of 
type $A_\ell$,
$d_0\leq\ell-2$ for $G$ of type $B_\ell$ or $C_\ell$, and $d_0\leq \ell-3$ for
$G$ of type $D_\ell$.
Then the weight $\mu-\alpha_{d_0}$ is dominant; indeed,
$$\mu-\alpha_{d_0}=\sum_{i=1}^{d_0-2} 
a_i\lambda_i+(a_{d_0-1}+1)\lambda_{d_0-1}+(a_{d_0}-2)\lambda_{d_0}+
\lambda_{d_0+1}.$$
Moreover $e(\mu-\alpha_{d_0})=e(\mu)$ and $\mu-\alpha_{d_0}$ occurs
with nonzero multiplicity in the irreducible $kG$-module with highest weight 
$\mu$.
\end{lemma}

\begin{proof} Note that we interpret $\lambda_0$ as $0$ in the above formula
if $d_0 = 1$. 

That $\mu-\alpha_{d_0}$ has the given form and is therefore
a dominant weight is easily verified. Also \cite[1.5(c)]{Seitzclass}
shows that $\mu-\alpha_{d_0}$ occurs with nonzero multiplicity in $V_G(\mu)$.

Now if $G$ 
is not of
type $A_\ell$, that $e(\mu-\alpha_{d_0})=e(\mu)$ follows directly from the 
definition. For $G$ of type $A_\ell$, since $d_0<\frac{\ell}{2}$, either $\ell$ is even 
and $d_0\leq\frac{\ell-2}{2}$ or $\ell$ is odd and $d_0\leq\frac{\ell-1}{2}$.
A direct calculation establishes the equality
$e(\mu-\alpha_{d_0})=e(\mu)$.\end{proof}

\begin{lemma}\label{lem:wt1} Let 
$\mu=\sum_{i<d} a_i\lambda_i +\lambda_d\in X(T_G)$ be a restricted 
dominant
weight, $\mu\ne\lambda_d$, with $d<\frac{\ell}{2}$ if $G$ is of type $A_\ell$ 
and $d<\ell-2$ otherwise. Then there exists a weight $\nu\in X(T_G)$,
 subdominant to 
$\mu$,
of the form $\nu=\sum_{j<d}b_j\lambda_j+\lambda_{d+1}$, with $e(\nu)=e(\mu)$.
Moreover, the weight $\nu$ occurs with nonzero multiplicity in the irreducible 
$kG$-module with highest weight $\mu$.
\end{lemma}

\begin{proof} Set $d_0={\rm max}\{i<d\ |\  a_i\ne0\}$. Set 
$\nu=\mu-\alpha_{d_0}-\alpha_{d_0+1}-\cdots-\alpha_d$. Then as in the above
 lemma
one verifies that $\nu$ is dominant and satisfies $e(\nu)=e(\mu)$. Moreover,
\cite[8.6]{Seitzclass} shows that $\nu$ occurs with nonzero multiplicity in $V_G(\mu)$.\end{proof}

\begin{lemma}\label{lem:wt2} Let 
$\mu=\sum_{i<d}a_i\lambda_i+\lambda_d\in X(T_G)$ be a restricted 
dominant weight, 
$\mu\ne\lambda_d$. Assume further that $d<\frac{\ell}{2}$ if $G$ is of type 
$A_\ell$, and $d<\ell-2$ otherwise.
Set $m_0={\rm min}\{e(\mu), \lfloor\frac{\ell}{2}\rfloor\}$, if $G$ is of type 
$A_\ell$, and 
set $m_0={\rm min}\{e(\mu), \ell-3\}$, otherwise. Then for each
$d\leq m\leq m_0$, there exists
a weight $\nu_m$, subdominant to $\mu$ of the form 
$\nu_m=\sum_{i<d}b_i\lambda_i+\lambda_m$, with $e(\nu_m)=e(\mu)$. Moreover,
each of the weights $\nu_m$ occurs with nonzero multiplicity in the 
irreducible $kG$-module
with highest weight $\mu$.
\end{lemma}

\begin{proof} We proceed by induction on $m$, noting that if $m=d$, we may 
take $\nu_m=\mu$. If $m=d+1$, the result follows from Lemma~\ref{lem:wt1}. For 
$m_0\geq m>d+1$, we suppose that there exists a weight 
$\nu_{m-1}=\sum_{j<d}b_j\lambda_j+\lambda_{m-1}$, subdominant to $\mu$ 
with $e(\nu_{m-1})=e(\mu)$. Note that $\nu_{m-1}\ne\lambda_{m-1}$, else
$e(\nu_{m-1})=m-1=e(\mu)$, contradicting $e(\mu)\geq m_0 \geq m$. Moreover 
$m-1<\frac{\ell}{2}$ if $G$ is of type 
$A_\ell$ and $m-1<\ell-2$ otherwise. So by Lemma~\ref{lem:wt1}, there exists
a weight $\nu_m=\sum_{j<d}b_j\lambda_j+\lambda_m$, subdominant to 
$\nu_{m-1}$, and therefore subdominant to $\mu$, with $e(\nu_{m})=e(\mu)$.
The existence of the subdominant weight $\nu_m$ now follows by induction.

To see that the weight $\nu_m$ occurs with nonzero multiplicity in the irreducible
$kG$-module $V_G(\mu)$, one must invoke the main result of \cite{Premet88}
and note that in all cases, the weights are obtained within a composition factor
for a Levi subgroup of type
$A$ whose root system has base 
$\{\alpha_1,\ldots,\alpha_{\lfloor\frac{\ell}{2}\rfloor}\}$ 
(respectively $\{\alpha_1,\dots, \alpha_{\ell-2}\}$), for $G$ of type $A_\ell$ 
(resp. not of type $A_\ell$).\end{proof}

\begin{lemma}\label{lem:wt3} Let $\mu\in X(T_G)$ be a restricted 
dominant weight with 
$r_\mu=0$; moreover assume 
$e(\mu)< \frac{\ell+2}{2}$ if $G$ is of type $A_\ell$, and 
$e(\mu)<\ell-2$ otherwise.
 Then the fundamental dominant weight 
$\lambda_{e(\mu)}$ is 
subdominant to $\mu$. Moreover, $\lambda_{e(\mu)}$ occurs with nonzero 
multiplicity
in the irreducible $kG$-module with highest weight $\mu$.
\end{lemma}

\begin{proof} Set $\mu=\sum_{i\leq m}a_i\lambda_i$, where 
$m\leq\frac{\ell}{2}$ and $a_m\ne0$. If $\mu=\lambda_m$, then $e(\mu)=m$ and 
the result holds.

If $a_m>1$, then $2m\leq e(\mu)\leq\frac{\ell+1}{2}$. Note that if $\ell<4$,
then $\lambda = 2\lambda_1$ and $\lambda-\alpha_1$ is the desired subdominant weight. So we may assume 
$\ell\geq 4$, so we have $m<\frac{\ell-1}{2}$ if $G=A_\ell$ and $m\leq\ell-2$
otherwise; hence by Lemma~\ref{lem:old1}, 
$\mu-\alpha_m$ is subdominant to $\mu$ and $e(\mu-\alpha_m)=e(\mu)$. Moreover,
if $a_m>1$ and $\mu-\alpha_m$ is a fundamental dominant weight then
$\mu=2\lambda_1$, $\mu-\alpha_1=\lambda_2$ which is indeed $\lambda_{e(\mu)}$ 
as desired.

So we now suppose without loss of generality the existence of a subdominant 
weight $\nu$ satisfying $\nu=\sum_{i\leq m}b_i\lambda_i+\lambda_d$, where 
$m<d\leq\frac{\ell+2}{2}$, $\nu\ne\lambda_d$ and $e(\nu)=e(\mu)$.

Consider first the case where $G$ is of type $B_\ell$, $C_\ell$ or $D_\ell$.
Here we have $e(\nu)<\ell-2$ and so $d<\ell-2$. By Lemma~\ref{lem:wt2}, 
there exists a subdominant weight 
$\nu_r$ of the form $\nu_r=\sum_{i<d}c_i\lambda_i+\lambda_r$, where 
$r={\rm min}\{e(\nu),\ell-3\}$ and $e(\nu_r)=e(\nu)=e(\mu)$. But our assumption
$e(\mu)\leq\ell-3$ implies that $r=e(\mu)$ and so $\nu_r=\lambda_r$ is 
the desired weight. Moreover, as in the previous proof, the weights are 
obtained via the action of a Levi subgroup of type $A$, and hence occur
with nonzero multiplicity in the irreducible $kG$-module with highest weight $\mu$, by \cite{Premet88}.

We now turn to the case where $G$ is of type $A_\ell$.
Here we have $e(\nu) = e(\mu)<\frac{\ell+2}{2}$. It suffices to establish
the existence of the subdominant weight, as \cite{Premet88} then yields the 
result.

Suppose $\ell$ is odd and $e(\nu)=\frac{\ell+1}{2}$. Then $e(\nu)>d$ implies that 
$d<\frac{\ell}{2}$. Moreover, 
$r={\rm min}\{e(\nu),\lfloor\frac{\ell}{2}\rfloor\}=\frac{\ell-1}{2}$ and so by 
Lemma~\ref{lem:wt2}, 
there exists a weight $\eta$, subdominant to $\nu$, of the form
$\eta=\sum_{i<d}c_i\lambda_i+\lambda_{(\ell-1)/2}$, with 
$e(\eta)=e(\nu)=\frac{\ell+1}{2}$. Since 
$e(\lambda_{(\ell-1)/2})=\frac{\ell-1}{2}$, 
$\eta\ne\lambda_{(\ell-1)/2}$. So $c_i\ne0$ for some $i<d$.
Take $t$ maximal such that $c_t\ne0$. Then 
$\rho=\eta-(\alpha_t+\alpha_{t+1}+\cdots+\alpha_{(\ell-1)/2})$ is subdominant to $\eta$
and one checks that
$e(\rho)=e(\eta)=e(\nu)=\frac{\ell+1}{2}$.
But this then implies that $\rho=\lambda_{(\ell+1)/2}$, satisfying the conclusion
of the result.

Now suppose $\ell$ is even and $e(\nu)=\frac{\ell}{2}$. If $d=\frac{\ell+2}{2}$
or $\ell/2$, then $e(\nu)>\ell/2$. So we have that $d<\ell/2$. Moreover, 
$r={\rm min}\{e(\nu),\lfloor\frac{\ell}{2}\rfloor\}=
e(\nu)=\frac{\ell}{2}$ and by Lemma~\ref{lem:wt2}, there exists a weight 
$\eta$, subdominant to $\nu$ of the form 
$\eta=\sum_{i<d}d_i\lambda_i+\lambda_{\ell/2}$, 
with $e(\eta)=e(\mu)=\frac{\ell}{2}$. But this then implies that $\eta=
\lambda_\frac{\ell}{2}$, which gives the result.

The two limiting cases having been handled, we can now assume
that $e(\nu)<\frac{\ell}{2}$, so $d<\frac{\ell}{2}$ and 
$r={\rm min}\{e(\nu),\lfloor\frac{\ell}{2}\rfloor\}=e(\nu)$, and 
by Lemma~\ref{lem:wt2}, there exists a 
weight $\nu_r$, subdominant
to $\nu$ of the form $\nu_r=\sum_{i<d}d_i\lambda_i+\lambda_r$ with 
$e(\nu_r)=e(\nu)$. But $r=e(\nu)$ then gives $\nu_r=\lambda_r$, the desired 
subdominant
fundamental weight.\end{proof}

We can now establish a bound for $\dim V_G(\lambda)$ in terms of $\dim W$ 
and $e(\lambda)$.

\begin{prop}\label{easybounds} Let $\lambda\in X(T_G)$ be a restricted 
dominant weight, set $e: = e(\lambda)$, 
and assume $e\leq\frac{\ell+1}{2}$ if $G$ is of type $A_\ell$, respectively 
$e\leq\ell-3$, otherwise. 
Then
 $$\dim V_G(\lambda)\geq {{\ell+1}\choose e}, \mbox{ if }G\mbox{ is of type }A_\ell,$$ and
$$\dim V_G(\lambda)\geq 2^e{\ell\choose e},\mbox{ if } G\mbox{ is of type } 
B_\ell, C_\ell,\mbox{ or } D_\ell.$$ 
Moreover, in all cases,
$$\dim V_G(\lambda)\leq(\dim W)^e.$$
\end{prop}

\begin{proof} We first establish the upper bound. Recall that the irreducible 
$kG$-module with highest weight $\lambda_i$ occurs as a composition factor in $\Lambda^iW$, and therefore
as a composition factor of $W^{\otimes i}$, 
for $i\leq\frac{\ell+1}{2}$, if $G$ is of type $A_\ell$, and for $i\leq\ell-3$
otherwise.  
For $G$ of type $A_\ell$, the $kG$-module of highest weight $\lambda_i$, 
$i>\frac{\ell+1}{2}$,
occurs as a composition factor of $\Lambda^{\ell-i+1}(W^*)$,
and therefore as a composition factor of $(W^*)^{\otimes(\ell-i+1)}$. 
Now if $G$ has type $B_\ell$, $C_\ell$ or $D_\ell$ and $e\leq\ell-3$,
then $\lambda=\sum_{i=1}^{\ell-3} a_i\lambda_i$. Combining all of these remarks,
we see that
$V_G(\lambda)$ 
occurs as a composition factor of the $kG$-module 
$\otimes_{i=1}^\ell(\Lambda^i W)^{\otimes a_i}$ 
(or $\otimes_{i=1}^{l_\lambda}(\Lambda^i W)^{\otimes a_i}
\otimes\otimes_{i=1}^{r_\lambda}(\Lambda^{\ell-i+1} W^*)^{\otimes a_{\ell-i+1}}$,
where $l_\lambda$ and $r_\lambda$ are as in Definition~\ref{e(lambda)}).
We then conclude that $\dim V_G(\lambda)\leq(\dim W)^e$, as desired.

Now turn to the lower bound. 
If $\lambda$ is a fundamental dominant 
weight, then Lemma~\ref{Weylgrp-stab} gives the desired lower bound.
So we suppose $\lambda$ is not a fundamental dominant weight. If $G$ is of type 
$A_\ell$, write $\lambda=\lambda_\ell+\lambda_r$, where 
$\lambda_\ell=\sum_{i\leq l_\lambda}a_i\lambda_i$ and 
$\lambda_r=\sum_{i> l_\lambda}b_i\lambda_i$. 

If $G$ is of type $B_\ell$, $C_\ell$ or $D_\ell$, or if $G$ is of type 
$A_\ell$ and $\lambda=\lambda_\ell$ or 
$\lambda=\lambda_r$,
 then Lemma~\ref{lem:wt2} implies
that $\lambda_e$ is a subdominant weight occurring in the irreducible module 
$V_G(\lambda)$, (or in
$V_G(\lambda)^*$)
and hence we conclude as above. In the remaining case, when $G$ is of type 
$A_\ell$ and 
$\lambda \ne\lambda_\ell$ and $\lambda\ne\lambda_r$,
we apply Lemma~\ref{lem:wt2}
to each of the weights $\lambda_\ell$ and $\lambda_r$. Note that 
$e = e(\lambda_\ell)+e(\lambda_r)$ and hence
we find that $\lambda_{e(\lambda_r)} +\lambda_{e(\lambda_\ell)}$ is a 
subdominant weight
occurring in the irreducible $kG$-module $V_G(\lambda)$. Here we use Lemmas~\ref{lem:bc-app}(2) and 
\ref{Weylgrp-stab} to
see that 
$\dim V_G(\lambda)\geq \frac{(\ell+1)!}{e(\lambda_\ell)!e(\lambda_r)!(\ell+1-e)!}\geq{{\ell+1}\choose{e}}$.\end{proof}

\begin{lemma}\label{lowerbd1} Let $G$ be of type $A_\ell$, and
let $\lambda\in X(T_G)$ be a restricted dominant weight. If 
$e(\lambda)>(\ell+1)/2$, then 
$\dim V_G(\lambda)\geq {{\ell+1}\choose{\lfloor\ell/2\rfloor}}$.

\end{lemma} 

\begin{proof} By \cite{Premet88}, all weights subdominant to $\lambda$ occur 
with
nonzero multiplicity in $V_G(\lambda)$. It suffices to exhibit a weight
such that the number of $W_G$-conjugates is at least 
${{\ell+1}\choose{\lfloor\ell/2\rfloor}}$.
Let 
$\lambda =  \mu_1 +\mu_2$, where 
$\mu_1 = \sum_{i=1}^{\lfloor(\ell+1)/2\rfloor} a_i\lambda_i$ and 
$\mu_2 = \sum_{j=\lceil(\ell+2)/2\rceil}^\ell b_j\lambda_j$. 

If $e(\mu_i)<\frac{\ell+2}{2}$, for $i=1,2$, then we apply Lemma~\ref{lem:wt3} 
to each of the weights
$\mu_i$, $i=1,2$ to obtain a subdominant weight with nonzero coefficients
 of $\lambda_{e(\mu_1)}$
and $\lambda_{\ell+1-e(\mu_2)}$. Hence 
$$\dim V_G(\lambda) 
\geq \frac{(\ell+1)!}{e(\mu_1)!e(\mu_2)!(\ell+1-e(\mu_1)-e(\mu_2))!}.$$ 
But since $e(\mu_1)+e(\mu_2) = e(\lambda)>\frac{\ell+1}{2}$, we may apply 
Lemma~\ref{lem:bc-app} with $m=\ell+1$ to obtain the 
desired inequality. 

Suppose now that $e(\mu_1)\geq\frac{\ell+2}{2}$, and 
in particular, $\mu_1$ is not a fundamental dominant weight.
Let $m$ be maximal such that $a_m\ne 0$. If $m \geq\lfloor\ell/2\rfloor$, 
then the weight $\lambda$
has at least ${{\ell+1}\choose{\lfloor\ell/2\rfloor}}$ conjugates, so we 
may assume
$m<\ell/2$. Now if $a_m=1$, we may use Lemma~\ref{lem:wt2} to produce a 
subdominant
weight with nonzero coefficient of $\lambda_{\lfloor\ell/2\rfloor}$, 
which again gives the desired lower bound. So finally consider the case where 
$a_m>1$. Then $$\nu: = \lambda - \alpha_m = 
\sum_{i=1}^{m-2}a_i\lambda_i+(a_{m-1}+1)\lambda_{m-1}
+(a_m-2)\lambda_m+\lambda_{m+1}+\mu_2$$ is a subdominant weight with 
$e(\nu) = e(\lambda)$. Now $m+1\leq\frac{\ell+1}{2}$; if 
$m+1\geq\ell/2$, there are at least 
${{\ell+1}\choose{\lfloor\ell/2\rfloor}}$ distinct conjugates of $\nu$ and 
if $m+1<\ell/2$, we argue as above.

Finally, we must consider the case where $e(\mu_2)\geq\frac{\ell+2}{2}$.
But this is completely analogous.\end{proof}

\begin{lemma}\label{lowerbd2} Let $G$ be of type $B_\ell$, $C_\ell$ or $D_\ell$,
with $\ell\geq 3$. Let $\lambda\in X(T_G)$ be a restricted dominant weight,
$\lambda = \sum_{i=1}^\ell a_i\lambda_i$. Set $e: = e(\lambda)$.
The following hold:
\begin{enumerate} [a)]

\item If $r_\lambda\ne0$,
then $\dim V_G(\lambda)\geq 2^{\ell-1}$.

\item If $r_\lambda = 0$, and $e\leq (\ell+1)/2$, then 
$\dim V_G(\lambda)\geq {\ell\choose e}$. If moreover $\ell\geq 7$, then 
$\dim V_G(\lambda)\geq 2^{e}{\ell\choose e}$.
\item If $r_\lambda=0$ and $e>(\ell+1)/2$, then 
$\dim V_G(\lambda)\geq 2^{\ell-1}$.
\end{enumerate}
\end{lemma}

\begin{proof} For (a), we first note that if $a_\ell\ne 0$,
or if $a_\ell+a_{\ell-1}\ne 0$ when $G=D_\ell$, then
the lower bound follows directly from Lemma~\ref{Weylgrp-stab}.
So now assume $a_\ell = 0$ and in addition $a_{\ell-1}=0$ if $G=D_\ell$.
Again applying Lemma~\ref{Weylgrp-stab}, we have
$\dim V_G(\lambda) \geq 2^{\ell-r_\lambda+1}{\ell\choose {r_\lambda-1}}$. Now 
since $r_\lambda\leq \ell/2$, we have that 
${\ell\choose {r_\lambda-1}}\geq 2^{r_\lambda-1}$ and so 
$\dim V\geq 2^{\ell}$.

The bound of (b) follows from Proposition~\ref{easybounds} as long
as $\frac{\ell+1}{2}\leq\ell-3$. If $\ell=3,4$ or 5, so $e \leq
2,2,3$ respectively, then the minimal
dimension of an irreducible $kG$-module exceeds ${\ell\choose e(\lambda)}$.
If $\ell=6$ and so $e\leq 3$, then it is a direct check to see
that there exists a subdominant weight whose orbit under the Weyl group
has length exceeding ${\ell\choose e(\lambda)}$.

Finally, for (c) assume $r_\lambda=0$ and $e(\lambda)>\frac{\ell+1}{2}$.
We can argue exactly as in the proof of the previous lemma 
(in the case where $e(\mu_1)\geq\frac{\ell+2}{2}$) to
 produce a subdominant weight with $r_\lambda\ne0$ and then use (a).\end{proof}

\end{section}

\begin{section}{Preliminary analysis}\label{sec:prelim}

Let $H = H_r(q)$ be a quasisimple finite classical group in characteristic
$p$, let $k$ be an algebraically closed field of characteristic different 
from $p$. Let
$W$ be a nontrivial 
$kH$-module. 
Let $G$ be the smallest classical group containing the image of the 
corresponding representation of $H$.
We maintain the notation established in Sections
~\ref{sec:intro} and ~\ref{sec:notation}. 
In particular, $P<H$ is the stabilizer of a singular $1$-space of $N$,
the natural module for $H$, and $P=QL$, where $Q=O_p(P)$. We let 
$\chi\in Q^* = {\rm Hom}(Q/[Q,Q],k^*)$ and set 
$W_\chi=\{w\in W\mid xw=\chi(x)w\mbox{ for all } x\in Q\}$. 
Let $P_\chi:=\Stab_P(W_\chi)$; then
$P_\chi=QL_\chi$, where $L_\chi = \Stab_L(W_\chi)$. 
Write $L_\chi^\infty$ for the last term in the derived series of $L_\chi$.

It is useful to collect some information about the structure of $P$, 
$L_\chi$ and the orbit sizes of $L$
acting on $Q^*$. We record this information in Table~\ref{tab:HP}.

\begin{table}[!h]
$$\begin{array}{|l|l|l|l|l|} \hline
&&&&\\
H& Q & L' & \mbox{ generic structure of } L_\chi&[L:L_\chi] \\
&&&&\\ \hline
&&&&\\
\Lin_{r+1}(q)& q^r& \Lin_r(q)& q^{r-1}:\Lin_{r-1}(q)& q^r-1\\
&&&&\\
\hline
&&&&\\

\U_{r+1}(q)&q^{1+2(r-1)}&\U_{r-1}(q)&q^{1+2(r-3)}:\U_{r-3}(q)& 
(q^{r-1}+(-1)^r)(q^{r-2}-(-1)^r)\\
&&&&\\
&&&\U_{r-2}(q)&(q-1)q^{r-2}(q^{r-1}-(-1)^{r-1})\\
&&&&\\
\hline
&&&&\\

\Sp_{2r}(q)& q^{1+2(r-1)}& \Sp_{2(r-1)}(q)&q^{1+2(r-2)}:\Sp_{2(r-2)}(q)& 
q^{2(r-1)}-1\\
q\mbox{ odd }&&&&\\
&&&&\\
\hline
&&&&\\
{\mathbf O}_{2r+1}(q)& q^{2r-1}& {\mathbf O}_{2r-1}(q)& q^{2r-3}:{\mathbf O}_{2r-3}(q)&q^{2r-2}-1
 \\
&&&&\\
&&&\Om^+_{2r-2}(q)&\frac{(q-1)}{2}(q^{2r-2}+q^{r-1})\\
&&&&\\
&&&\Om^-_{2r-2}(q)&\frac{(q-1)}{2}(q^{2r-2}-q^{r-1})\\
&&&&\\
\hline
&&&&\\

\Om^+_{2r}(q)& q^{2r-2}& \Om^+_{2r-2}(q)&q^{2r-4}:\Om^+_{2r-4}(q)&q^{2r-3}+q^{r-1}-q^{r-2}-1\\
&&&&\\
&&&{\mathbf O}_{2r-3}(q)&(\frac{q-1}{2-\delta_{2,p}})(q^{2r-3}-q^{r-2})\\
&&&&\\
\hline
&&&&\\
\Om_{2r}^-(q)& q^{2r-2}& \Om^-_{2r-2}(q)&q^{2r-4}:\Om^-_{2r-4}(q)&q^{2r-3}-q^{r-1}+q^{r-2}-1 \\
&&&&\\
&&&{\mathbf O}_{2r-3}(q)&(\frac{q-1}{2-\delta_{2,p}})(q^{2r-3}+q^{r-2})\\
&&&&\\
\hline

\end{array}$$
\caption{Orbit size and structure}
\label{tab:HP}
\end{table}

\begin{rem}
\begin{enumerate}[1)]
\item The $L$-orbit of singular vectors in $Q^*$ corresponds precisely
to the cases where the stabilizer $L_\chi'$ satisfies $O_p(L_\chi')\ne 1$.
\item When $H=\U_{r+1}(q)$, there are exactly $q-1$ distinct $L'$-orbits of 
nonsingular vectors, all of which are fused by the action of the central torus
of $L$.

\item When $H={\mathbf O}_n(q)$, the situation is more complicated. 
 When $p \neq 2$, or $p = 2$ and $H$ is of type ${\mathbf O}_{2r+1}(q)$, then 
there are 
2 $L'$-orbits of nonsingular 1-spaces 
which implies that there are
2 $L$-orbits of nonsingular vectors
on $Q$ (the natural module for $L'$), 
whereas there is a unique 
$L'$-orbit on nonsingular 1 spaces when 
$p= 2$ and $H$ is of type $\Om_{2r}^\pm(q)$
leading to a single $L$-orbit on $Q$. 

\end{enumerate}
\end{rem}

\begin{lemma}\label{lem:basecases} Let $H$ and $W$ satisfy the $Q$-linear large 
hypotheses. Then one of the following holds.
\begin{enumerate}[a)]
\item $H=\Lin_{4}(q)$ and $q\geq 4$;
\item $H=\Lin_{r+1}(q)$ and $r\geq 4$;
\item $H = \U_{5}(q)$ and $q\geq 4$;
\item $H = \U_{6}(q)$ and $q\geq 3$;
\item $H = \U_{r+1}(q)$ and $r\geq 6$;
\item $H = \Sp_6(q)$ and $q\geq 5$;
\item $H = \Sp_{2r}(q)$ and $r\geq 4$;
\item $H={\mathbf O}_{2r+1}(q)$ and $r\geq 3$;
\item $H =\Om^\pm_{2r}(q)$ and $r\geq 4$.
\end{enumerate}

\end{lemma}

\begin{proof} Note that the $Q$-linear large hypothesis implies that
$L_\chi$ is nonsolvable. Relying upon the information in Table~\ref{tab:HP},
it is a direct check to see that the only finite quasisimple classical groups
omitted are those for which $L_\chi$ is solvable.\end{proof}

We record in Table~\ref{tab:bB-ll}
 the Landazuri-Seitz \cite{LaSe}, Seitz-Zalesski \cite{SeZa},
 and Guralnick-Tiep \cite{GT_low} bounds
for the minimal dimension of an irreducible cross-characteristic representation of 
$H$, noted $b_1$, for each
of the classical groups $H$. In addition, we provide a lower bound of a simpler 
($q$-power) form.
In the final column, we record an upper bound $B$ for the dimension of such a
representation, 
deduced from \cite[Thm. 2.2]{SeitzCC}. 

\begin{table}[!h]
$$\begin{array}{|llll|} \hline
&&&\\
H& r&b_1 & B \\
&&&\\ \hline
&&&\\
\Lin_{2}(q)&& \frac{q-1}{{\rm gcd}(q-1,2)}& q+1 \\
&&&\\

\Lin_{r+1}(q)&r\geq 2& \frac{q^{r+1}-1}{q-1}-2>q^r+q^{r-1}& q^{\frac{r(r+3)}{2}} \\
&&&\\
\hline
&&&\\
\U_{r+1}(q)&r\geq 2\mbox{ even }&q(q^r-1)/(q+1)>q^r-q^{r-1}&q^{\frac{r(r+3)}{2}}\\

&&&\\

&r\geq 3\mbox{ odd }&(q^{r+1}-1)/(q+1)>q^r-q^{r-1}&q^{\frac{r(r+3)}{2}}\\

&&&\\\hline
&&&\\

\Sp_{2r}(q)&r\geq 2& (q^r-1)/2> q^{r-1}& q^{r^2+r}\\
&&&\\\hline

&&&\\

{\mathbf O}_{2r+1}(q)&&&\\
q \mbox{ odd }, q\ne 3&r\geq 3 &\frac{q^{2r}-1}{q^2-1}-r> q^{2r-2}& q^{r^2+r} \\
&&&\\
 q=3&r\geq 3& \frac{3^{2r}-1}{3^2-1}-\frac{3^r-1}{2}> q^{2r-2}& q^{r^2+r} \\
&&&\\
 q \mbox{ even }&r\geq 2& \frac{q(q^r-1)(q^{r-1}-1)}{2(q+1)}\geq q^{2r-3}& q^{r^2+r} \\
&&&\\
\hline
&&&\\

\Om^+_{2r}(q)&r\geq 4&&\\
q\ne 2,3&& \frac{q(q^{2r-2}-1)}{q^2-1}+q^{r-1}-r>q^{2r-3}& q^{r^2} \\
&&&\\
q=2,3&&             \frac{q(q^{2r-2}-1)}{q^2-1}-\frac{q^{r-1}-1}{q-1}-7\delta_{2,q}>q^{2r-3}&q^{r^2}\\
&&&\\\hline
&&&\\

\Om^-_{2r}(q)&r\geq4& \frac{q(q^{2r-2}-1)}{q^2-1}-q^{r-1}-r+2>q^{2r-3} & q^{r^2} \\
&&&\\

 \hline
\mbox{Exceptions}&&&\\
 b_1(\Lin_2(4)) = 2;&b_1(\Lin_2(9)) = 3;&b_1(\Lin_3(2)) = 2;\qquad b_1(\Lin_3(4)) = 4;&b_1(\Lin_4(2)) = 7;\\
b_1(\Lin_4(3)) = 26;&b_1({\mathbf O}_5(2)) = 2;&b_1(\U_4(2)) = 4;\quad\  b_1(\U_4(3)) = 6;&b_1(\Om_8^+(2)) = 8;\\
b_1(\O_7(3)) = 27.&&&\\
\hline
\end{array}$$
\caption{Definition of the bounds $b_1$ and $B$}
\label{tab:bB-ll}
\end{table}

Hence we have

\begin{lemma}\label{lem:bB} Let $U$ be a nontrivial irreducible 
$kH$-module. Then 
$b_1\leq \dim U\leq B$.
\end{lemma}

In addition, we now establish a better lower bound for $W$, given that we are working 
under the assumption
that $H$ and $W$ satisfy the $Q$-linear large hypothesis.

\begin{prop}\label{b2-bound} Let $H$ and $W$ satisfy the $Q$-linear large hypothesis. 
Assume moreover 
that 
$O^p(L_\chi^\infty)$ is not one of the exceptions in
Table~\ref{tab:bB-ll}. Then 
$\dim W\geq b_2$, where $b_2$ is as recorded in Table~\ref{tab:b2}.
\end{prop}
\begin{table}[!h]
$$\begin{array}{|lll|} \hline
&&\\
H&&b_2 \\
&&\\ \hline
&&\\
\Lin_4(q)&q\ne 4,9&\frac{q-1}{{\rm gcd}(q-1,2)} (q^3-1)\\
&&\\
\Lin_{r+1}(q)&r\geq 4& (\frac{q^{r-1}-1}{q-1}-2)(q^r-1) \\
&&\\
\hline

&&\\
\U_5(q)&q\ne 4,9&\frac{q-1}{{\rm gcd}(q-1,2)}(q-1)q^2(q^3+1)\\
&&\\
\U_6(q)&q\ne 4,9&\frac{q-1}{{\rm gcd}(q-1,2)}(q^7+q^4-q^3-1)\\
&&\\
\U_{r+1}(q)&r\geq6&\\
&r \mbox{ even}&\frac{q^{r-3}-q}{q+1}(q^{2r-3}-q^{r-1}+q^{r-2}-1)\\
&&\\
&r\mbox{ odd}&\frac{q^{r-3}-1}{q+1}(q^{2r-3}+q^{r-1}-q^{r-2}-1)\\
&&\\
\hline

&&\\

\Sp_{2r}(q)&r\geq 3& \frac{q^{r-2}-1}{2}(q^{2(r-1)}-1) \\
&&\\
\hline
&&\\
{\mathbf O}_7(q)&q\ne 2,3, 4,9&\frac{q-1}{{\rm gcd}(q-1,2)}(q^4-1)\\
&&\\
{\mathbf O}_{2r+1}(2)&r\geq 5&\frac{2(2^r-1)(2^{r-1}-1)}{6}(2^{2r-2}-1)\\
&&\\

{\mathbf O}_{2r+1}(q)&r\geq 4&(\frac{q^{2r-4}-1}{q^2-1}-\frac{q^{r-2}-1}{q-1})(q^{2r-2}-1)\\

&q>2&\\
 \hline
&&\\
\Om^+_8(q)&q\ne 2,4,9&\frac{q-1}{{\rm gcd}(q-1,2)}(q^5+q^3-q^2-1)\\
&&\\
\Om^+_{2r}(2)&r\geq 7&(\frac{(2^{r-2}-1)(2^{r-3}-1)}{3}(2^{2r-3}-2^{r-2})\\
&&\\
\Om^+_{2r}(q)&r\geq 5&(\frac{q(q^{2r-6}-1)}{q^2-1}-\frac{q^{r-3}-1}{q-1}-7)(q^{2r-3}+q^{r-1}-q^{r-2}-1)\\
&q>2&\\
\hline
&&\\
\Om^-_8(q)&q\geq 4&\frac{q^2-1}{{\rm gcd}(q^2-1,2)}(q^5-q^3+q^2-1)\\
&&\\
\Om^-_{10}(q)&q\geq 4&\frac{q^4-1}{q+1}(q^7-q^4+q^3-1)\\
&&\\
\Om^-_{2r}(q)&r\geq6&(\frac{q(q^{2r-6}-1)}{q^2-1}-q^{r-3}-r+4)(q^{2r-3}-q^{r-1}+q^{r-2}-1) \\
&&\\
\hline

\end{array}$$
\caption{Generic lower bound $b_2$ for $Q$-linear-large modules}
\label{tab:b2}
\end{table}

\begin{proof} Since $W = C_W(Q)\oplus (\sum_{\chi\in Q^*, \chi\ne 1} W_\chi)$, and there exists 
$\chi\ne 1$
such that ${\rm soc}(W_\chi|_{L_\chi^\infty})$ has an irreducible $L_\chi^\infty$-submodule 
$S$ of dimension greater 
than $1$, we may obtain a lower bound for $\dim W$ as the product of the 
lower bound for a nontrivial
irreducible $kL_\chi^\infty$-module, as given in Table~\ref{tab:bB-ll}, and
the size of a minimal orbit of $L$ acting on $Q^*$, that is the minimal value for 
$[L:L_\chi]$
as in Table~\ref{tab:HP}.\end{proof}

\begin{rmk}\label{b_2-exc} The groups for which some $O^p(L_\chi^\infty)$ is one of the 
exceptions in
Table~\ref{tab:bB-ll}, and which are therefore
not covered by Proposition~\ref{b2-bound}, are as follows:

\noindent (i)\ $H = \Lin_4(q)$, $q=4,9$;\ (ii)\ $H = \Lin_5(q)$, $q=2,4$;\ 
(iii)\ $H = \Lin_6(q)$, $q=2,3$;\\
(iv)\ $H = \U_{r+1}(q)$, $r = 4,5$, $q=4,9$;\ (v)\ $H = \U_{r+1}(q)$, 
$r=6,7$, $q=2,3$;\ (vi)\ $H = \Sp_6(9)$;\ (vii)\ $H = {\mathbf O}_7(q)$, $q=2,3,4,9$;\ 
(viii)\ $H = {\mathbf O}_n(q)$, $9\leq n\leq 11$, $q=2,3$;\\ (ix)\ $H = \Om_8^+(q)$, 
$q=2,4,9$;\ (x)\ $H = \Om_8^-(q)$, $q=2,3$;\ (xi)\ $H = \Om_{12}^+(2)$.
\end{rmk}

\begin{lemma}\label{lem:b2} Let $H$ and $W$ satisfy the $Q$-linear large 
hypothesis. 
Assume moreover that 
$O^p(L_\chi^\infty)$ is not one of the exceptions in
Table~\ref{tab:bB-ll}, that is $H$ is not one of the groups in 
Remark~\ref{b_2-exc}. Let $b_2$ be as in Table~\ref{tab:b2} and $B$ as in 
Table~\ref{tab:bB-ll}. Then 
 either $2^{\lceil(b_2-1)/2\rceil}>B^3$ or $H={\mathbf O}_7(2)$ or $H = {\mathbf O}_7(3)$.
\end{lemma}

\begin{proof} Suppose that $H\not \in\{{\mathbf O}_7(2), {\mathbf O}_7(3)\}$. 
Taking $\log_2$ of both sides of the desired inequality, we see that it
 suffices to show $b_2 > 6\log_2B+1$. Then it is a direct check that $b_2$ exceeds $6\log_2B+1$ in each
of the remaining cases. \end{proof}


\begin{lemma}\label{lem:B3} Let $H$ and $W$ satisfy the $Q$-linear large 
hypothesis and assume moreover that $H\notin\{\Lin_4(4),\Lin_5(2),\Lin_6(2),
{\mathbf O}_7(2),{\mathbf O}_9(2), \Om_8^\pm(2)\}$. 
Let $V$ be an irreducible 
$kH$-module. Then $2^{\lceil(\dim W-1)/2\rceil}>(\dim V)^3$.
\end{lemma}

\begin{proof} If $O^p(L_\chi^\infty)$ is not one of the exceptions in
Table~\ref{tab:bB-ll}, that is, if $H$ is not one of the groups
listed in Remark~\ref{b_2-exc} and  if $H\not\in\{{\mathbf O}_7(2), {\mathbf O}_7(3)\}$, 
the result follows from Lemmas~\ref{lem:b2} and ~\ref{lem:bB}. So we must 
consider each
of the exceptions in turn. 

We first simply apply the bound $\dim W\geq b_1(H)$ and find
that $2^{\lceil(\dim W-1)/2\rceil}>B^3$,
where $B$ is the upper bound for the dimension of an irreducible $kH$-module
unless $H$ is one of the following groups: $\Lin_4(4)$, $\Lin_5(2)$,
$\Lin_6(2)$, $\U_7(2)$, $\U_8(2)$, ${\mathbf O}_7(2)$, ${\mathbf O}_7(3)$, $\Om_8^\pm(2)$, 
${\mathbf O}_9(2)$, $\Om_{10}^\pm(2)$, and ${\mathbf O}_{11}(2)$. We now apply
a better lower bound, as in the proof of Lemma~\ref{lem:b2},
using the $Q$-linear large hypothesis, to see that 
$\dim W\geq b_1(O^p(L_\chi^\infty))[L:L_\chi]$, 
choosing of course a nonsolvable $O^p(L_\chi^\infty)$ and again find that 
$2^{\lceil(\dim W-1)/2\rceil}>B^3$ unless $H$ is one of the groups listed as 
exceptions in the result.\end{proof}

\begin{cor}\label{rlambda} Let $H$ and $W$ satisfy the $Q$-linear large 
hypothesis and assume moreover
that $H$ fixes a nondegenerate from on $W$; so if $G\subset {\rm Isom}(W)$
is the smallest classical group containing the image of $H$, $G$ is of type 
$B_\ell$, $C_\ell$ or $D_\ell$.  Moreover, assume that 
$H\ne {\mathbf O}_7(2)$. Let 
$V_G(\lambda)$ be an irreducible $kG$-module
with restricted highest weight $\lambda$, on which $H$ acts irreducibly.
 Then $r_\lambda = 0$.
\end{cor}

\begin{proof} Suppose the contrary, that is, $r_\lambda\ne 0$. 
Note that Lemma~\ref{lem:basecases} implies that $\ell\geq 3$.

Then by Lemma~\ref{lowerbd2}, 
$\dim V_G(\lambda)\geq 2^{\ell-1}$. Since 
$\ell = \lceil\frac{\dim W-1}{2}\rceil$, we have 
$\dim V_G(\lambda)\geq 2^{\lceil\frac{\dim W-1}{2}\rceil-1}$. 
Assume for the moment that $H\notin\{\Lin_4(4),\Lin_5(2),\Lin_6(2),
{\mathbf O}_9(2), \Om_8^\pm(2)\}$. Then by 
Lemma~\ref{lem:B3}, we have $$ 2^{\lceil\frac{\dim W-1}{2}\rceil-1} = 
(1/2)2^{\lceil\frac{\dim W-1}{2}\rceil}>
(1/2)(\dim V_G(\lambda))^3>\dim V_G(\lambda),$$ providing a contradiction.

Consider now the six groups omitted above. We argue as in the proof of
Proposition~\ref{b2-bound} to obtain a lower bound for $\dim W$ and hence a lower bound for $\ell$. Moreover, we use a tighter upper bound
for the dimension of a nontrivial irreducible $kH$-module, as indicated
in Table~\ref{tab:betterB}, in the following section. In each of the remaining cases, $2^{\ell-1}>\dim V_G(\lambda)$, providing the required contradiction.\end{proof}

\end{section}


\begin{section}{The restricted case}

In this section, we establish the main theorem
under the additional hypothesis that $V = V_G(\lambda)$ has a restricted 
highest weight $\lambda$. Recall that $V$ is a nontrivial
 irreducible $kH$-module. Initially, we will not require the hypothesis that $W$ is irreducible; indeed
up through Lemma~\ref{lem:5cases} we will assume only that $W$ is a $Q$-linear large $kH$-module.
We will point out which point we assume the irreducibility hypothesis.
Throughout this section let $e: = e(\lambda)$ and let $S$ be a nontrivial irreducible 
constituent
of ${\rm soc}(W_\chi)|_{L_\chi^\infty}$, for some $\chi\in Q^*$, $\chi\ne1$, as in 
Definition~\ref{linlarge}.

\begin{prop}\label{prop:e-bd} Let $\ell$ be the rank of the classical group 
$\Cl(W)$.\begin{enumerate}[]  
\item[\rm a)] If $H \ne {\mathbf O}_7(2)$, then $e\leq(\ell+1)/2$.
\item[\rm b)] If $H \not\in\{{\mathbf O}_7(2), \Om_8^+(2)\}$, then $e\leq 2$ or
$e< \dim S$.
\end{enumerate}
\end{prop}

\begin{proof} For (a), suppose the contrary:
$e>\frac{\ell+1}{2}$. In case $\Cl(W) = \SL(W)$, Lemmas~\ref{lem:basecases}
 and \ref{lem:bB}
imply that $\ell\geq 5$. Then Lemmas~\ref{bc-lower} and \ref{lowerbd1}
show that $\dim V>2^\ell/(\ell+1)\geq2^\ell/\ell^2$. In case $\Cl(W)$ is of type 
$B_\ell, C_\ell$ or $D_\ell$, then Lemma~\ref{lowerbd2} gives 
$\dim V\geq 2^\ell/\ell^2$.
Now $\ell\geq \lceil\frac{\dim W-1}{2}\rceil$ and $\ell\leq B$, so 
$\dim V\geq 2^{\lceil\frac{\dim W-1}{2}\rceil}/B^2$, which by Lemma~\ref{lem:B3}
exceeds $B$, contradicting Lemma~\ref{lem:bB}, unless $H$ is one of $\Lin_4(4)$,
$\Lin_5(2)$, $\Lin_6(2)$, ${\mathbf O}_9(2)$, or $\Om_8^\pm(2)$.

In the remaining cases,
we consult the modular character tables in \cite{GAP}
and see that for $d$ the minimal dimension of an irreducible $kH$-module with 
$d\geq\dim W$,
the value $2^{\lceil(d-1)/2\rceil}$ exceeds the cube of the 
maximal-dimensional irreducible $kH$-module. 
This contradicts our inequality $\dim V\geq 2^\ell/\ell^2$, hence establishing
the inequality.

For (b), we assume that $e\geq 3$.
For the moment, assume in addition that $H$ is not one of the following groups 
(which are treated at the end of the proof):

\noindent (i)\ $\Lin_4(q)$, $q=4,5,7,9$;\ (ii)\ $\Lin_5(q)$, $q=2,4$;\ (iii)\ 
$\Lin_6(2)$;\\
(iv)\ $\U_5(q)$, $q=4,5,7,8,9,11$;\ (v)\ $\U_6(3)$;\ (vi)\ $\U_7(q)$, $q=2,3$;\ 
(vii)\ $\U_8(q)$, $q=2,3$;\\ 
(viii)\ ${\mathbf O}_{2r+1}(q)$, $r=4,5$, $q=2,3$;\ (ix)\ ${\mathbf O}_7(q)$, $q= 3,4,5,7,9,11,13$;\\
(x)\ $\Om_{2r}^\pm(q)$, $r=5,6$, $q=2,3$;\ (xi)\ $\Om_8^-(q)$, $q=2,3$;\ 
(xii)\ $\Om_8^+(q)$, $q=3,4,5,7,8$.

(These groups are excluded from our first analysis for two reasons:
either $O^p(L_\chi^\infty)$ provides an exception to the general polynomial 
bounds for $\dim W$
(see Table~\ref{tab:bB-ll}), or the 
polynomial bounds are 
too small because $q$ and $r$ are small.)

Now suppose $e \geq\dim S$. By (a),  $e\leq{(\ell+1)}/2$, and by 
Corollary~\ref{rlambda}, if $G = B_\ell, C_\ell$ or $D_\ell$, then 
$r_\lambda = 0$. Then Proposition~\ref{easybounds} and Lemma~\ref{lowerbd2}
 imply
that $\dim V\geq {{\ell+1}\choose e}$ for $G=A_\ell$,
and $\dim V\geq 2^e{\ell\choose e}$, for $G=B_\ell$, $C_\ell$ or $D_\ell$, 
the latter as long as $\ell\geq 7$. Applying Lemmas~\ref{lem:basecases} and ~\ref{lem:bB}, we see that $\dim W> 13$ and so $\ell\geq 7$. 




Set $a = \lceil([L:L_\chi]-1)/2\rceil$
and $s = \dim S$, so that $e\geq s$. Since $\dim V\geq[L:L_\chi]\dim S$, we have $\ell\geq as$ 
and so

$${\ell\choose e}\geq {\ell\choose s}\geq{{as}\choose s}.$$

Now if $s<a$, then $s^2<as$ and by Lemma~\ref{lem:bc-app},
 we have ${{as}\choose s}\geq (as)^{s/2}$.
If on the other hand $s\geq a$ so that we have 
$(\ell+1)/2\geq e\geq s\geq a$, then $a^2\leq as$ and we have 
$${\ell\choose e}\geq {{as}\choose s}\geq{{as}\choose a}\geq (as)^{a/2}.$$
Hence we now have 
$$(as)^m\leq \dim V\leq B, \mbox{ where }m = {\rm min}\{a/2,s/2\} 
\mbox{ and }B\mbox{ is as in Table~\ref{tab:bB-ll}}.$$

Taking $\log_q$ of this inequality gives
\begin{equation}\label{ineq:asB} m(\log_q(as))\leq\log_q B.
\end{equation}

Now we refer to Table~\ref{tab:HP} for a lower bound for $a$, as well as the
structure of $O^p(L_\chi^\infty)$. Now $s$ is greater than or equal to the 
Landazuri-Seitz-Zalesski, Guralnick-Tiep
bound for the group $O^p(L_\chi^\infty)$ (see Table~\ref{tab:bB-ll}; 
denote this bound by 
$b_1(O^p(L_\chi^\infty))$.
Moreover,  investigation of these
lower bounds shows that $a\geq b_1(O^p(L_\chi^\infty))$ in every case; 
hence we replace
the inequality (\ref{ineq:asB}) by 

\begin{equation}\label{ineq:asB2} 
\frac{b_1(O^p(L_\chi^\infty))}{2}\cdot\log_q(a\cdot b_1(O^p(L_\chi^\infty)))\leq\log_q B.
\end{equation}

Our assumption that $H$ is not one of the groups in (i) - (xii) above 
implies that 
there are no solutions
to inequality (\ref{ineq:asB2}). 

We now consider the cases of (i) - (xii), where we argue that
the exceptional multipliers do not generally appear in the group 
$O^p(L_\chi^\infty)$,
and hence we can revert to the polynomial lower bound for the minimal
dimension of a nontrivial irreducible $O^p(L_\chi^\infty)$ cross-characteristic 
representation. Moreover, we provide a more precise,
improved upper bound for $\dim V$, deduced from \cite[Thm. 2.2]{SeitzCC}.

By Lemmas~\ref{lem:basecases} and \ref{lem:bB}, we see that $\ell\geq 7$. 
In particular, $e\leq\frac{\ell+1}{2}\leq\ell-3$ and so we may apply 
Proposition~\ref{easybounds} to see that $\dim V\geq {\ell+1\choose e}$, 
if $G=A_\ell$ and $\dim V\geq 2^e{\ell\choose e}$, if $G=B_\ell,C_\ell$ or 
$D_\ell$.  Note as well that as $e\geq 3$, we have 
$2^e{(c-1)/2\choose e}<{{c-1}\choose e}$ if $c$ is odd,
and $2^e{c/2\choose e}<{{c-1}\choose e}$ if $c$ is even. 
 Hence we use in every case the lower bound for 
$\dim V$
\begin{equation}\label{ineq:smcases}
\dim V\geq 2^e{{\lceil\frac{\dim W-1}{2}\rceil}\choose e}
\end{equation}
and we take as our lower bound for $\dim W$, as in previous arguments,
the product of $[L:L_\chi]$ and the minimal dimension of a 
nontrivial cross-characteristic representation of $O^p(L_\chi^\infty)$.
For the upper bound, we do not use the approximations in 
Table~\ref{tab:bB-ll},
but rather the actual upper bounds given by \cite[Thm. 2.1]{SeitzCC}. We record
  these values for each of the groups in Table~\ref{tab:betterB}.

\begin{table}[!h]
$$\begin{array}{|l|l|} \hline
&\\
H& \mbox{ upper bound for } \dim V \\
&\\ \hline
&\\
\Lin_{r+1}(q)& \frac{(q^{r+1}-1)(q^r-1)\cdots(q^2-1)}{(q-1)^r}\\
&\\
\hline
&\\

\U_{r+1}(q), r\mbox{ odd}&\frac{(q^{r+1}-(-1)^{r+1})(q^r-(-1)^r)\cdots(q^2-1)}{(q^2-1)^{(r+1)/2}} \cdot\frac{q^2-1}{q^2-q+1}\\
&\\
\U_{r+1}(q), r\mbox{ even}&\frac{(q^{r+1}-(-1)^{r+1})(q^r-(-1)^r)\cdots(q^2-1)}{(q+1)(q^2-1)^{r/2}} \cdot\frac{q^2-1}{q^2-q+1}\\
&\\
\hline
&\\
{\mathbf O}_{2r+1}(q)& \frac{(q^{2r}-1)(q^{2r-2}-1)\cdots(q^2-1)}{(q-1)^r}\\
&\\
\hline
&\\

\Om^+_{2r}(q)& \frac{(q^{r}-1)(q^{2r-2}-1)(q^{2r-4}-1)\cdots(q^2-1)}{(q-1)^r}\\
&\\
\hline
&\\
\Om_{2r}^-(q)& \frac{(q^{r}+1)(q^{2r-2}-1)(q^{2r-4}-1)\cdots(q^2-1)}{(q+1)(q-1)^{r-1}} \\
&\\

\hline

\end{array}$$
\caption{Upper bound for $\dim V$}
\label{tab:betterB}
\end{table}

For $H = \Lin_4(q)$, $q=4,5,7,9$, we first note that 
$O^p(L_\chi^\infty) \supset\SL_2(q)$, whose minimal 
dimensional 
nontrivial representation over $k$
has dimension $q-1$ if $q$ is even, and $(q-1)/2$ if $q$ is odd. So we 
have 
$$\frac{(q^4-1)(q^3-1)(q^2-1)}{(q-1)^3}\geq\dim V
\geq 2^e{{\frac{(q^3-1)(q-1)-{\rm gcd}(2,q-1)}{2{\rm gcd}(2,q-1)}}\choose e}.$$

One now checks that for $e\geq 3$, and $q=4,5,7,9$, the above inequality is 
never satisfied.

For $H=\Lin_5(4)$, the minimal dimensional module for $O^p(L_\chi^\infty)$ is
of dimension at least 4. Hence, $\dim W\geq (4^4-1)\cdot 4$ 
and one checks that 
$2^e{509\choose e}$
exceeds $(q^5-1)(q^4-1)(q^3-1)(q^2-1)/(q-1)^4$, when $q=4$ and $e\geq 3$, 
contradicting 
the inequality 
(\ref{ineq:smcases}). 
For $H=\Lin_6(2)$, a similar analysis gives a contradiction. For the group
$H = \Lin_5(2)$ a slightly finer analysis is required. Here we
see that $O^p(L_\chi^\infty)\supset\SL_3(2)$ and the minimal
dimension of a cross-characteristic representation of this group is $3$.
Hence $\dim W\geq 3(2^4-1)$, while consulting the modular character tables in 
\cite{GAP}, we see that the maximal possible dimension for $V$ is 1240.
These values contradict the inequality (\ref{ineq:smcases}). 

Turn now to the unitary groups. For $\U_5(q)$, where we must consider the cases 
$4\leq q\leq 11$, we have 
$\dim W\geq (q^5+q^2)\cdot b_1(\U_2(q))= (q^5+q^2)\cdot\frac{q-1}{{\rm gcd}(2,q-1)}$. (As above, we have that $O^p(L_\chi^\infty)\supset \SL_2(q)$ in all cases
and so this is indeed an accurate lower bound.)
A direct calculation
shows that $2^e{{\lceil (\dim W-1)/2\rceil}\choose e}$ exceeds the upper 
bound for $\dim V$ as given in Table~\ref{tab:betterB}.
For $H = \U_6(3)$, we may use the lower bound of Table~\ref{tab:b2},
$\dim W\geq b_2\geq (3^7+3^4-3^3-1)$. 
Again, a direct check 
gives a contradiction
to the inequality (\ref{ineq:smcases}) and the upper bound of 
Table~\ref{tab:betterB}.

For $H=\U_7(q)$, $q=2,3$, since $\U_3(2)$ is solvable, we have 
$\dim W\geq (2^9+2^4)b_1(\U_4(2)) = 4(2^9+2^4)$,
when $q=2$ and $\dim W\geq (3^9-3^5+3^4-1)b_1(\U_3(3)) = 6(3^9-3^5+3^4-1)$, 
when $q=3$. In each case, we 
obtain a contradiction by comparing the upper and lower bounds for $\dim V$.
A similar analysis for the groups $H=\U_8(q)$, $q=2,3$, 
yields a contradiction as well.

For ${\mathbf O}_7(q)$, with $q=3,4,5,7,8,9,11,13$, we have 
$\dim V\leq \frac{(q^6-1)(q^4-1)(q^2-1)}{(q-1)^3}$, while
$$\dim W\geq {\rm min}\{(q^4-1)b_1({\mathbf O}_3(q)), (q^4-q^2)b_1(\Om_4^-(q)), 
(q^4+q^2)b_1(\Om_4^+(q))\}.$$
Hence $\dim W\geq {\rm min}\{(q^4-1)\frac{q-1}{{\rm gcd}(q-1,2)}, 
(q^4-q^2)\frac{q^2-1}{{\rm gcd}(q-1,2)}\}$.
Note that we use the second of these terms in the case $q=3$, since 
${\mathbf O}_3(3)$ is solvable. Now one
checks that the lower bound given by (3) exceeds the upper bound in every case. 

For $H = \Om_8^+(q)$, with $q=3,4,5,7,8$, we have $\dim V\leq  
\frac{(q^4-1)(q^6-1)(q^4-1)(q^2-1)}{(q-1)^4},$
while $\dim W\geq{\rm min}\{(q^5+q^3-q^2-1)b_1(\Om_4^+(q)), 
(q^5-q^2)b_1({\mathbf O}_5(q))\}.$ 
Note that we use
the latter expression if $q=3$, since $\Om_4^+(3)$ is solvable.
Now $b_1(\Om_4^+(q)) = \frac{q-1}{{\rm gcd}(q-1,2)}$ and $b_1({\mathbf O}_5(q))=
\frac{q^2-1}{2}$ if $q$ is odd, 
respectively, 
$\frac{q(q^2-1)(q-1)}{2(q+1)}$ for $q$ even. In particular, 
$\dim W\geq (q^5+q^3-q^2-1)(q-1)/{\rm gcd}(q-1,2)$ 
in all cases.
Now one checks that $2^e{{\lceil(\dim W-1)/2\rceil}\choose e}$ exceeds 
$\dim V$ for all $e\geq 3$. 

For $H={\mathbf O}_9(q)$, $q=2,3$, we have $$\dim W\geq{\rm min}\{(q^6-1)b_1({\mathbf O}_5(q)), 
(q^6+q^3)/(1+\delta_{2,p})b_1(\Om_6^+(q)), 
(q^6-q^3)/(1+\delta_{2,p})b_1(\Om_6^-(q))\}.$$ 
Recall that $b_1(\Om^+_6(q)) = b_1(\Lin_4(q))$ and 
$b_1(\Om_6^-(q)) = b_1(\U_4(q))$. Note that $O^p(L_\chi^\infty)$ contains 
a subgroup isomorphic to either the orthogonal group ${\rm O}_5(2)$, or the
linear group ${\rm SL}_4(2)$, or the unitary group ${\rm SU}_4(2)$. 
We now consult the modular character tables
in \cite{GAP} to find that the minimal dimensional
representations of these groups are $7,7$ and $5$, respectively, and hence 
$\dim W\geq {\rm min}\{7(2^6-1), 7(1/2)(2^6+2^3), 5(1/2)(2^6-2^3)\}$, while 
(again using \cite{GAP}) the maximal
dimension of a nontrivial cross-characteristic representation of ${\mathbf O}_9(2)$ is 
$68850$.
This contradicts the inequality (\ref{ineq:smcases}). For $q=3$, we use an 
analogous argument. 

For $H = \Om^\pm_{10}(q)$, with $q=2,3$, we have 
$$\dim W\geq{\rm min}\{(q^7\pm q^4\mp q^3-1)b_1(\Om_6^\pm(q)), 
(q^7\mp q^3)b_1({\mathbf O}_7(q))\}.$$ Now we use the fact that 
$O^p(L_\chi^\infty)\supset\SL_4(q)$, 
respectively $\SU_4(q)$, 
when $H$ is of $+$, 
respectively $-$ type.
We can then consult the modular character tables in \cite{GAP} to obtain the 
precise lower bounds for the dimensions of the irreducible
representations of these groups. We have $b_1(\SL_4(2)) = 28$, 
$b_1(\SU_4(2)) = 34$, and 
$b_1(\SU_4(3)) = 15$.
Hence $\dim W\geq (q^7\pm q^4\mp q^3-1)b_1(\Om_6^\pm(q))$. 
We use the upper bound 
$\dim V\leq 68850$, which is found in the 
character table
for $H$. Now one
can check  that in each case inequality (\ref{ineq:smcases}) fails.

For $H = {\mathbf O}_{11}(q)$, $q=2,3$, we have  
$\dim V\leq \frac{(q^{10}-1)(q^8-1)(q^6-1)(q^4-1)(q^2-1)}{(q-1)^5}$,
while $\dim W\geq{\rm min}\{(q^8\pm q^4)b_1(\Om_8^\pm(q))/(1+\delta_{2,q}), 
(q^8-1)b_1({\mathbf O}_7(q))\}.$ Using that 
$O^p(L_\chi^\infty)$ contains the almost simple group $\Om_8^+(2)$, and not 
a double
 cover,
when $p=2$, we see that we may use the lower bound $b_1(\Om_8^+(2)) = 28$. 
(See the modular
character tables in \cite{GAP}.) It is then
a direct check to see that in each case the inequality (\ref{ineq:smcases}) fails.

Finally, for $H=\Om_{12}^{\pm}(q)$, $q=2,3$, we have 
$\dim V\leq \frac{(q^6\mp1)(q^{10}-1)(q^8-1)(q^6-1)(q^4-1)(q^2-1)}{(q\mp1)(q-1)^5}$.
Here $\dim W\geq{\rm min}\{(q^9\mp q^4)b_1(\Om_9(q)), 
(q^9\pm q^5\mp q^4-1))b_1(\Om_8^\pm(q))\}$. Again,
one can work with the almost simple group $\Om_8^+(2)$, and in the other
cases use the general polynomial bound for $b_1(O^p(L_\chi^\infty))$. 
In every case, 
inequality (\ref{ineq:smcases})
fails.\end{proof}

\begin{rem} It is perhaps worth noting that the analysis of the previous proof does not require
that $S$ be an $L_\chi^\infty$-submodule of $W_\chi$, but rather simply that we have a nonlinear irreducible
composition factor. This observation could be useful when one considers the case when $W$ is 
not a $Q$-linear large $kH$-module.
\end{rem}

\begin{lemma}\label{newupper} Assume that $\dim S >e$ and that $e\leq(\ell+1)/2$. 
Then there 
exists
a $P_\chi$-invariant submodule of $V$ of dimension at most $(\dim S)^e$.
In particular, $\dim V\leq(\dim S)^e[H:P_\chi]$.
\end{lemma}

\begin{proof} By 
Proposition~\ref{prop:e-bd}, we have that $e\leq(\ell+1)/2$ and $e<\dim S$
and Corollary~\ref{rlambda} assures that
either $r_\lambda = 0$ or $G$ is of type 
$A_\ell$.
Hence, if $G$ is of type $A_\ell$, we have 
$\l=\sum_{i=1}^{\dim S-1}a_i\l_i+\sum_{i=\ell-\dim S+2}^\ell a_i\l_i$ and 
$\lambda=\sum_{i=1}^{\dim S-1}a_i\l_i$, otherwise.

We first consider the case where $S$ is a totally singular 
subspace of $W$.
Then ${\rm Stab}_G(S)$ is a parabolic 
subgroup $P_G=Q_GL_G$, with $[L_G,L_G]\cong \GL(S)\times\Isom(S^\perp/S)$.
If $r_\lambda l_\lambda=0$, then
without loss of generality, we take the $\GL(S)$ factor to be the 
subsystem subgroup of $G$ corresponding to the set of simple roots 
$\{\a_1,\dots,\a_{\dim S-1}\}$ (or $\{\a_{\ell},\dots,\a_{\ell-\dim S+2}\}$
if $r_\lambda\ne0$).
In this case, $C_V(Q_G)$ is a $P_G$-invariant submodule on which, 
by \cite{Smith},
$L_G$ acts irreducibly with highest weight $\lambda|_{L_G}$.
Moreover, 
$P_\chi\leq\Stab_G(S)=P_G$ and so $P_\chi$ also acts on $C_V(Q_G)$. 
The irreducibility of $V|_H$ then implies that 
$\dim V\leq \dim C_V(Q_G)[H:P_\chi]$. By Lemma~\ref{easybounds}, 
$\dim C_V(Q_G)\leq(\dim S)^e$, and the result follows.
In the case where $S$ is totally singular and $r_\lambda l_\lambda\ne0$ (so $G$
is of type $A_\ell$), write
$\lambda = \lambda_\ell+\lambda_r$, where $\lambda_\ell$ has support
on $\{\alpha_1,\dots,\alpha_{l_\lambda}\}$ and $\lambda_r$ has support
on $\{\alpha_{\ell-r_\lambda+1},\dots,\alpha_\ell\}$. In particular,
$e(\lambda) = e(\lambda_\ell)+e(\lambda_r)<\dim S$ and
we replace the $\GL(S)$ factor in the above argument by the 
subsystem subgroup of $G$ corresponding to the set of simple roots 
$\{\a_1,\dots,\a_{e(\lambda_l)},\a_{\ell},\dots,\a_{\ell - \dim S-e(\lambda_l)-1}\}$.
Now apply Lemma~\ref{easybounds} once again to obtain the result. 

We now consider the case where $S$ is not totally singular.
In particular, $G$ preserves a nondegenerate symplectic or quadratic form, so 
$r_\lambda=0$,
and the existence of the nontrivial $Q$-character $\chi$ implies that $p=2$,
and so ${\rm char}(k)\ne 2$. 
Hence, the irreducibility of $S|_{L_\chi}$ implies that $S$ is a 
nondegenerate subspace with 
respect to the bilinear form on $W$. In this case, we have 
$P_\chi<({\rm Isom}(S)\times {\rm Isom}(S^\perp))\cap G$. 
As in the previous case,
we may assume without loss of generality that the restriction of $\l$ to a 
maximal torus of ${\rm Isom}(S^\perp)$ is the zero weight. Moreover, taking a 
standard embedding of $({\rm Isom}(S)\times {\rm Isom}(S^\perp))\cap G$
in $G$, we have that the
maximal vector $v^+$ of $V$ with respect to a fixed Borel subgroup of $G$
is a maximal vector for a Borel subgroup of ${\rm Isom}(S)$ with the highest 
weight 
given by restricting $\l$ to the standard torus of ${\rm Isom}(S)$. We have 
the $(({\rm Isom}(S)\times {\rm Isom}(S^\perp))\cap G)$-invariant submodule
$Y:=(({\rm Isom}(S)\times {\rm Isom}(S^\perp))\cap G).v^+$,
which is an image of the Weyl module for ${\rm Isom}(S)$ of highest weight $\l$.
Hence, as in the proof of Lemma~\ref{easybounds}, $\dim Y\leq (\dim S)^e$.\end{proof}

\begin{prop}\label{lin-la-gen} Assume that 
$H\not\in\{{\mathbf O}_7(2), {\mathbf O}_9(2), \Om_8^+(2), \Lin_5(2), \Lin_6(2)\}$.
Then $e=1$ or $e=2$.
\end{prop}

\begin{proof} We assume $e\geq 3$ and arrive at a contradiction.
By Proposition~\ref{prop:e-bd}, we have $\dim S >e$ and $e\leq(\ell+1)/2$, and hence 
we may apply Lemma~\ref{newupper}.

We continue with the notation as established above.
In particular, we have $\dim W\geq[L:L_\chi]\dim S$; hence we have 
$\dim W>13$ and so $\ell\geq 7$.

Now note that for all natural numbers $s,a,e\geq1$, with $e\leq a$, we have
$${sa\choose e}=
\frac{(sa)(sa-1)\cdots(sa-e+1)}{e!} = 
s^e\frac{a(a-1/s)\cdots(a-(e-1)/s)}{e!}\geq s^e{a\choose e}.$$ 
Set $s=\dim S$ and $a=\lceil([L:L_\chi]-1)/2\rceil$ if $G$ is of type 
$B_\ell$, $C_\ell$ or $D_\ell$
and let $a = [L:L_\chi]-1$ if $G$ is of type $A_\ell$. As 
$\ell = \lceil(\dim W-1)/2\rceil$, or $\ell=\dim W-1$, respectively, 
we have $\ell\geq sa$.

By Lemma~\ref{easybounds}, 
Proposition~\ref{prop:e-bd}, Lemma~\ref{newupper} and the above remarks, we have
 $$(\dim S)^e{a\choose e}\leq{{sa}\choose{e}}\leq {\ell\choose e}
\leq\dim V\leq (\dim S)^e[H:P_\chi],$$
if $G$ is of type $A_\ell$, and  $$(\dim S)^e2^e{a\choose e}\leq 2^e{{as}\choose e}\leq 2^e{\ell\choose e}
\leq\dim V\leq (\dim S)^e[H:P_\chi],$$ otherwise.
So we have
$${[L:L_\chi]-1\choose e}\leq [H:P_\chi],$$ or 
$$2^e{\lceil([L:L_\chi]-1)/2\rceil\choose e}\leq [H:P_\chi],$$ according to whether $G$ has type $A_\ell$ or not. Note that 
$[H:P_\chi]=[H:P][L:L_\chi]$. Also, as $e\geq 3$ and
$e\leq (\ell+1)/2$, we have
${\ell\choose e}\geq {\ell\choose 3}$.

We simplify the expressions $2^3{([L:L_\chi]-1)/2\choose 3}$ and 
${[L:L_\chi]-1\choose 3}$.
 We have 
\begin{equation}\frac{1}{6}[L:L_\chi]([L:L_\chi]-3)([L:L_\chi]-6) < 
2^3{([L:L_\chi]-1)/2\choose 3};
\end{equation}

and 

\begin{equation}\frac{1}{6}[L:L_\chi]([L:L_\chi]-2)([L:L_\chi]-4) < 
{[L:L_\chi]-1\choose 3}.
\end{equation}

Hence we must now solve the inequalities

\begin{equation}\label{ineq-bd-1}\frac{1}{6}([L:L_\chi]-3)([L:L_\chi]-6) \leq [H:P]
\end{equation} and 

\begin{equation}\label{ineq-bd-2}\frac{1}{6}([L:L_\chi]-2)([L:L_\chi]-4) \leq [H:P].
\end{equation}

The proposition follows from the next claim in case $e\leq a$; that is the 
claim shows that there
are no solutions to the inequalities (\ref{ineq-bd-1}) and (\ref{ineq-bd-2}), 
and hence
gives the desired contradiction.

\begin{cla} For every $H$ as in the statement of the proposition, we have 
$$([L:L_\chi]-3)([L:L_\chi]-6)>6[H:P].$$
\end{cla}

\noindent{\it Proof of claim}: 
Note that $[H:P]$ is the number
of singular $1$-spaces in $N$, the natural (projective) module for $H$ 
(and hence
can be deduced from the last column of Table~\ref{tab:HP}).
Comparing the degrees of the polynomial expressions obtained by 
using the minimal 
possible value
for $[L:L_\chi]$ and the actual value of $[H:P]$ gives one a clear idea that 
the claim is reasonable.
It is a tedious but straightforward exercise to verify the inequality for each 
of the groups $H$ covered  by the proposition. This completes the proof of the claim.\\

Now consider the case where $e>a$; that is, $e>[L:L_\chi]-1$, if $G$ is of 
type $A_\ell$ and $e>\lceil([L:L_\chi]-1)/2\rceil$, otherwise. Here we refer
to the proof of Proposition~\ref{prop:e-bd}; the argument given there
applies and shows that under the condition that $e>a$,
the general lower bound of $2^e{\ell\choose e}$ or ${\ell\choose e}$ exceeds
the upper bound $B$ for $\dim V$.

This completes the proof of the proposition.\end{proof}

\begin{lemma}\label{lem:5cases} Assume that $H$ is one of
the groups ${\mathbf O}_7(2)$, $\Om_8^+(2)$, ${\mathbf O}_9(2)$, $\Lin_5(2)$ or $\Lin_6(2)$,
 and that 
$W$ satisfies the $Q$-linear large hypothesis.
 Let $V$ be a restricted
irreducible $kG$-module with highest weight $\lambda$, on which $H$ acts irreducibly. Then 
$e(\lambda)\leq 2$.
\end{lemma}

\begin{proof}  Set $e = e(\lambda)$. We assume that $e\geq 3$, and arrive at a contradiction in each case.

We first consider the group $H = {\mathbf O}_7(2)$.
 The $Q$-linear large hypothesis is 
satisfied only if
$O^p(L_\chi^\infty)$ is isomorphic to $\Om_4^-(2)$, since in the other cases, 
$L_\chi$ is 
solvable. Here
the orbit length $[L:L_\chi]$ is $6$. Moreover, as the group $H$ may 
have an exceptional
Schur multiplier, we must allow for the possibility that $L_\chi^\infty$ is the
double cover of the group ${\rm Alt}_5$. 
We also note that the maximal 
possible 
dimension for $V$ is $720$ 
(see \cite{GAP})). 
Since $8{d\choose3}$ exceeds $720$
for all $d\geq 21$, we see that $\dim W<21$. Consulting the modular character 
tables 
for $H$ in \cite{GAP},
we see that the only $kH$-module
that is both $Q$-linear large and of dimension less than 21 is either of dimension 
15, when 
${\rm char}(k)\ne3$ or of dimension 14 when ${\rm char}(k)=3$;
moreover, we have that $G = B_7$ when ${\rm char}(k)\ne3$ and $G = D_7$, 
when ${\rm char}(k)=3$.
We also note that the module $W$ in each case is a module for 
${\rm Sp}_6(2)$ and not 
for the double cover. Hence,
since we know that the weight $\lambda$ satisfies $r_\lambda=0$,
 the module $V$ is a 
module for the group ${\rm SO}(W)$
and so $V$ is also a representation of the group ${\rm Sp}_6(2)$, and 
not the double 
cover.
Hence our upper bound now becomes $512$, rather than 720. 
Now as $16{7\choose 4}>512$, we see that $e=3$. We may refer to 
\cite{luebeck} to see 
that the only module
for $G=B_7$, when ${\rm char}(k)\ne2,3$, with restricted highest 
weight $\lambda$ 
satisfying $e(\lambda)=3$, 
$r_\lambda =0$ and 
$\dim V\leq 512$ is the module with highest weight $\lambda_3$, of 
dimension 455. Finally, we conclude by noting that there is no
irreducible $kH$-module of dimension 455.

In case ${\rm char}(k)=3$, the maximal dimension of an irreducible $kH$-module
is 405. 
Here $G=D_7$ and 
by \cite{luebeck}
we see that the only $kG$-module
when ${\rm char}(k)=3$, with restricted highest weight $\lambda$ satisfying 
$e(\lambda)=3$, 
$r_\lambda =0$ and 
$\dim V\leq 405$, is the module with highest weight $\lambda_3$, of 
dimension 364. 
But now we refer to \cite{GAP}
to see that there is no such $3$-modular irreducible representation of $H={\mathbf O}_7(2)$. 
This completes the consideration 
of the case $H={\mathbf O}_7(2)$.

Now consider the case $H=\Om_8^+(2)$; here $\dim V\leq 6075$. 
Here we must have $L_\chi^\infty = {\mathbf O}_5(2)$ and $[L:L_\chi] = 28$. Since
the $3$-modular representations of ${\mathbf O}_5(2)$ are of degrees 2, 3, 4, 6 and 9, 
we see 
that $\dim W\geq 2\cdot28$. But comparing this with
the $3$-modular character table of $H$, we conclude that $\dim W\geq 104$. 
But then 
$8{{\lceil (\dim W-1)/2\rceil}\choose 3}$ exceeds 6075. 

Hence we are in the situation where ${\rm char}(k)\ne 2,3$. Here the 
character degrees 
of ${\mathbf O}_5(2)$ are 
4, 5, 8, 9, and 10. So we have $\dim W\geq 4\cdot28 = 112$. 
But then 
$8{{\lceil (\dim W-1)/2\rceil}\choose 3}$ exceeds 6075.

For the group ${\mathbf O}_9(2)$, the minimal
degree of a $Q$-linear large representation is 118 and so $\ell\geq 59$.
Also Propositions~\ref{prop:e-bd} and ~\ref{easybounds}
give a lower bound for $\dim V$ which exceeds the maximal degree of a 
cross-characteristic 
irreducible
representation, which is $68850$.

For the group $\Lin_5(2)$, a similar argument applies,
given that the minimal dimension of a $Q$-linear large representation
is 94 and the maximal dimension is 1240.

Finally, for the group $\Lin_6(2)$, the relevant lower, respectively upper,
bounds are 217, resp. 41664 again leading to a contradiction. 
This completes the proof of the lemma.

All of the precise upper and lower bounds used above are obtained 
from \cite{GAP} if not mentioned otherwise.\end{proof}

Combining Proposition~\ref{lin-la-gen} and Lemma~\ref{lem:5cases},
we now see that if $V$ is a restricted irreducible $kG$-module with highest 
weight $\lambda$ on which 
$H$ acts irreducibly, then $e(\lambda)\leq 2$. Our aim now is to see that in 
fact $e(\lambda)=1$, that is,
$V$ is simply the module $W$ or its dual $W^*$. This result plays a key role 
when we 
consider the more
general case of nonrestricted weights in the next section. The proof
is based upon results of \cite{MaTi} and \cite{MMT}, which
treat the irreducibility of tensor squares, and the symmetric 
and alternating square of irreducible $kH$-modules. Thus, for the next two results
we assume that $W$ is a nontrivial irreducible $kH$-module.

\begin{lemma}
Let $W$ be a nontrivial irreducible $kH$-module and assume $W$ and $H$ satisfy 
the 
$Q$-linear large hypothesis.
Let $V$
be an irreducible $kG$-module with restricted highest weight $\lambda$ 
satisfying $e(\lambda)=2$. Then either
$V|_H$ is reducible or ${\rm char}(k) = 3$, $H= {\mathbf O}_{2r+1}(2)$, $r\geq 4$ or $H=\Om_{2r}^\pm(2)$, $r\geq 5$. Moreover,
$W$ is equipped with an $H$-invariant bilinear or quadratic form, and 
$V$ is the maximal-dimensional composition factor of $\Lambda^2(W)$ or $S^2(W)$.
\end{lemma}

\begin{proof} Our hypothesis implies that $V$ is either the heart (i.e.,
the maximal-dimensional $kG$-composition factor) of the 
symmetric- or alternating square of 
$W$, or that $V$ is the heart of $W \otimes W^\ast$. 
The content of Theorems 1.2 and 3.1 of \cite{MMT} is that 
$V|_H$ is reducible unless
one of the following holds:
\begin{enumerate}[(i)] 
\item $W$ is a Weil module for $H$ of type $\Sp$ or $\U$, or 
\item $H$ is ``small'' (precisely given in \cite[Table 3.1]{MM} or \cite[Table 3.2]{MMT}), or
\item $H$ is not of type $\Lin$ and
$q \in \{ 2,4,8 \}$.
\end{enumerate}
In the last case,
 a precise bound for $\dim V$ is given. The Weil modules
for the symplectic and unitary groups do not satisfy the $Q$-linear large
hypothesis; nor do the small linear group examples, listed in 
\cite[Table 3.1]{MM}. Thus \cite[1.2, 5.3, 5.5, 5.7, 5.8, 5.13,]{MMT} imply  
that $V|_H$ is reducible unless 
possibly $V$ 
is the heart of the symmetric or 
alternating square of $W$, $G$ is of type $B_\ell$, $C_\ell$ or $D_\ell$,
and moreover one of following is true:
\begin{enumerate}[a)]
\item $H= \U_{r+1}(2)$, and
$\dim V \leq 2^{r+1}(2^{r+1}-(-1)^{r+1})(2^{r}+(-1)^{r})(2^{r-2} +1)/3$.
\item $H= {\mathbf O}_{2r+1}(q)$ with $q \in \{4,8 \}$, or $q = 2$, $r\geq 4$ and 
${\rm char}(k) \neq 3$; moreover 
$\dim W \leq 2(q^{2r-1}-1)$.
\item $H= {\mathbf O}_{2r+1}(2)$, $r\geq 4$, and ${\rm char}(k) = 3$; moreover 
$\dim V \leq 2^{4r-8}(2^{2r}-1)(2^{2r-2}-1)/5$.
\item $H= \Om^\pm_{2r}(q)$ with $q \in \{2,4 \}$ and 
${\rm char}(k) \neq 3$ if $(r,q) =(4,2)$; moreover
$\dim V \leq \frac{1}{2}q^{2r-2}(q^r \mp 1)(q^{r-1} \pm 1)/(q-1)$.
\item $H= \Om^\pm_{2r}(2)$, $r\geq 5$, and ${\rm char}(k) = 3$; moreover
$\dim V \leq 2^{4r-6}(2^{2}\mp 1)(2^{2r-2}-1)(2^{r-2}\mp 1)/15$. 
\item $H = {\mathbf O}_7(2)$, ${\rm char}(k)=3$, and $\dim W=7$ or $8$.
\item $H = \Om_8^\pm(2)$, ${\rm char}(k) = 3$.
\end{enumerate}

The cases (c) and (e) are as in the statement of the result; we now
show that in cases (a), (b), (d), (f) and (g), the assumption that $W$ is $Q$-linear large leads
to a contradiction.

We first note that in cases (f) and (g), using the modular character tables in \cite{GAP}, 
it is straightforward
to check
that there is no $Q$-linear large representation for which 
the symmetric, alternating or tensor square is irreducible.

In case $H=\U_{r+1}(2)$ the $Q$-linear large hypothesis implies 
that $r > 5$. Now by Proposition~\ref{b2-bound},
$\dim W\geq b_2 \geq \beta:= 2^{r}(2^{r-2}-1)(2^{r-1}-1)/3$ and thus the 
upper bound for 
the dimension of $V$ given in (a) above 
is less than $ {\beta \choose 2} -2$, and hence $V|_H$ is reducible.

If $H$ is of type ${\mathbf O}_n$ or $\Om_{2r}$ as in (b) or (d), 
then our estimate for 
$b_2$ (see Table~\ref{tab:b2}) exceeds $\dim W$ or as above 
${b_2\choose 2}-2$ exceeds the 
upper bound
for $\dim V$,  unless  $H\in\{{\mathbf O}_9(2),{\mathbf O}_7(4)\}$,
when $b_2$ was not defined. Thus, as above we conclude that either 
$V|_H$ is reducible
or $H$ is one of the groups ${\mathbf O}_9(2)$, ${\mathbf O}_7(4)$. For the first 
of the two exceptional cases, 
we use \cite{GAP} to check that the alternating square of the three lowest degree nontrivial
characters is irreducible. However, these characters correspond to modules
which do not satisfy the $Q$-linear large hypothesis. For the larger character degrees, 
we may argue as above to see that $V|_H$ is reducible.

For ${\mathbf O}_7(4) \cong \Sp_6(4)$ the ordinary character table is in \cite{GAP}.
The maximal degree of an irreducible character is 371280 and thus 
if $V|_H$ is irreducible, 
$\dim W$ is bounded above by 862. 
Now Table~\ref{tab:HP} shows that the possible orbit lengths of $L_\lambda$
are 255, 408 and 360, respectively, with $L_\lambda = 4^3:{\rm Alt}_5$, ${\mathbf O}_4^+(4)$ and 
${\mathbf O}_4^-(4)$, respectively. The argument on page 398 in the proof of 
\cite[5.5]{MMT} showing that since the representation affording $W$
has no root group in its kernel, the length of the orbit of 
$L_\lambda$ is 408 or 360, applies here. 
Now we note that the minimal dimension of an ${\mathbf O}_4^+(4)$-, respectively
${\mathbf O}_4^-(4)$-module, is at least 3, which means that the dimension 
of a $Q$-linear large module is at least $360\cdot3 > 862$. 
Hence $V|_H$ is reducible, as required.\end{proof}

Next we strengthen Proposition 5.7 of \cite{MMT} to eliminate possibility 
(c) in the list above.
 
\begin{lemma} Let $H= {\mathbf O}_{2r+1}(2)$, $r\geq 4$, and ${\rm char}(k) = 3$.
If $W$ is a nontrivial irreducible $kH$-module and $V$ is a restricted
irreducible $kG$-module with highest weight $\lambda$ satisfying
$e(\lambda) = 2$, and such that $V|_H$ is 
irreducible, then 
$$\dim V \leq 2^{r-3}(2^{r-2}-1)(2^{2r}-1)(2^{2r-2}-1).$$ Moreover, 
if $W$ is a $Q$-linear large $kH$-module, 
then $V|_H$ is reducible.
\end{lemma}

\begin{proof}
 Let $P$ be the parabolic subgroup of $H$ which stabilizes a totally 
singular $2$-space of the natural module of $H$.
The unipotent radical $U$ of $P$ satisfies $[U,U] < Z(U)$ with $|[U,U]| = 2$
 and $|Z(U)| = 8$. Now using the Chevalley 
commutator relations and the fact that the field of definition of $H$ is 
of even characteristic, we see that the 
subgroup of $U$ which is generated by the long root subgroups of $U$ is an 
 extraspecial subgroup $E$ of $U$ of order 
$2^{1+2(2r-4)}$. 
Moreover, the normalizer of $E$ in $P$ contains $U$ as well as
$C:= S_3 \times \Omega^-_{2r-4}(2)$.
Now, as was asserted in the proof of Proposition 5.5 of \cite{MMT},
the orbits of $P$ on $Z(U)^*$ have lengths $1$, $3$ and $3$. The character 
corresponding to the orbit of length $1$
contains $[U,U]$ in its kernel. So as $U$ acts faithfully on $W$, it must be 
the case that $Z(U)$ acts on 
$W$ in such a way that at least one of the orbits of length 3 of $Z(U)^*$
 is represented. Labelling 
the characters of such an orbit by $\chi_i$, 
$1 \leq i \leq 3$, we then have $W_{\chi_i} \neq 0$.
The weight spaces  $W_{\chi_i}$ are equivalent $EC^{\infty}$-modules.
Now Proposition 2.3 of \cite{MMT} implies that $V$ contains a nontrivial 
$EC^{\infty}$-invariant vector $v$. (Note that as $q=2$, the $\chi_i$ are
self-dual.)
Now $kUCv$ is a $UC$-invariant submodule of $V$. Also 
$UC/EC^{\infty} \cong S_4$,
so $kUCv$ contains a 
$UC$-invariant submodule of dimension at most $3$. 
Thus $\dim V$ is bounded 
above by 
$$ 3[H:UC] = 2^{r-3}(2^{r-2}-1)(2^{2r}-1)(2^{2r-2}-1),$$ which proves our 
first claim.

Now assume $W$ is $Q$-linear large and recall that $V$ is the
 maximal-dimensional composition factor of $\Lambda^2(W)$ or $S^2(W)$. 
When $r\geq 5$, Proposition~\ref{b2-bound} implies
that $\dim W\geq b_2\geq(2^{2r-2}-1)2^{r-2}(2^{r-1}-1)$  
which implies that 
${\beta \choose 2}-2$ is larger than $\dim V$. 
When 
$r = 4$, the $3$-modular 
character table of $H$ is available in \cite{GAP} and we check directly that 
$H$ acts reducibly on $V$ whenever 
$W$ is $Q$-linear large. \end{proof}

We now strengthen Proposition 5.13 of \cite{MMT} to eliminate possibility 
(e) from the list of possible 
obstructions above.

\begin{lemma}  Let $H=\Om^\pm_{2r}(2)$, $r\geq 5$, and ${\rm char}(k) = 3$.
If $W$ is a nontrivial irreducible $kH$-module and $V$ is a restricted
irreducible $kG$-module with highest weight $\lambda$ satisfying
$e(\lambda) = 2$, and such that $V|_H$ is 
irreducible, then 
$$\dim V \leq 2^{2r-5}(2^r \mp 1)(2^{2r-2}-1)(2^{r-2}\mp 1).$$ Moreover, if 
$W$ is $Q$-linear large, then $V|_H$ is reducible.
\end{lemma}

\begin{proof} Let $P \subset H$ be the stabilizer of a singular $1$-space
 of the 
natural module
of $H$ with unipotent radical $U$, and let 
$C = \Omega_2^-(2) \times \Omega_{2r-4}^\mp(2) \cong S_3 \times C^{\infty}$ be 
the naturally defined subgroup
of the Levi factor $\Omega_{2r-2}^+(2)$ of $P$.

Now the unipotent radical $U$ has the structure of a $C$-module,
the sum of the natural module for $ \Omega_2^-(2)$ and the natural
module for $\Omega_{2r-4}^\mp(2)$. In particular,
there is a subgroup $K$ of $U$ corresponding to the first summand, 
which is an elementary abelian $2$-group, containing $UC$ in its normalizer.

  
Now $W|_K$ contains 
nontrivial weight spaces for each of the
three distinct weights in $K^* = {\rm Hom}(K/[K,K],k^*)$. 
Again Proposition 2.3 of \cite{MMT} yields a 1-dimensional $UC^{\infty}$- submodule in $V$
and thus, as in the proof of the previous lemma,
 a $3$-dimensional $UC$-invariant submodule of $V$.
 This yields
$$\dim V \leq 3[H:UC] = 2^{2r-5}(2^r \mp 1)(2^{2r-2}-1)(2^{r-2}\mp 1),$$ 
which is our claim. 

Now assume $W$ is $Q$-linear large and recall that $V$ is the
 maximal-dimensional composition factor of $\Lambda^2(W)$ or $S^2(W)$. 
Note that when $r\geq 7$,
$\dim W\geq b_2 \geq \beta := 
(2\frac{2^{2r-6}-1}{3}-2^{r-3}-r+4)(2^{2r-3}-2^{r-1}+2^{r-2}-1)$ and again 
${\beta \choose 2}-2$ is larger than $\dim V$, when $r \geq 7$. When 
$r = 5$, we consult the $3$-modular 
character tables in \cite{GAP} and then, as above, we check our assertion 
directly.

When $r=6$, then our bound for $\dim V$ yields that 
$\dim W \leq 17011$. Since $O^p(L_\chi^\infty)\supset \Omega_8^+(2)$,
we see that $\dim W\geq 495\cdot b$, where $b$ is the minimal
dimension of an irreducible $k\Omega_8^+(2)$-module. If this dimension is at
least 35, we have the desired contradiction. The only other
possibility is a $28$-dimensional $k\Omega_8^+(2)$-module. 
More precisely, we have to analyze the situation where 
$H = \Om_{12}^+(2)$, $[L:L_\chi]= 2^9 + 32-16-1 = 527$ and the socle of 
$W_\chi$ contains
the 28-dimensional module for $2^8:\Om_8^+(2)$; call this submodule 
$X\subset {\rm soc}(W_\chi|_{L_\chi})$.
As $q = 2$, the group $Q$ acts trivially on the alternating 
and symmetric squares of $W_\chi$. 

The alternating and symmetric squares of $X$ possess $28$-, respectively 
$35$-dimensional 
submodules. Thus the group $Q:2^8:\Om_8^+(2)$ has a $28$-, respectively $35$-dimensional 
submodule on $V$. Consequently $\dim V$ is bounded above by 
$28\cdot527\cdot(2^{11} +2^6-2^5-1)$ ($=28\cdot[H:QL_\chi]$),
respectively $35\cdot527\cdot2079$. But these numbers are smaller than 
${{28\cdot527}\choose2} - 2$, respectively ${{35\cdot527 + 1} \choose 2} -2$. 
So again we have our usual contradiction.\end{proof}

We are now ready to establish the main result of this section: 

\begin{thm}\label{thm_restricted} Let $W$ be a $Q$-linear large $kH$-module, 
$G\subset \GL(W)$
the smallest classical group containing the image of $H$ under the associated
representation, and let $V$ be an
irreducible $kG$-module with restricted highest weight $\lambda$. Then either 
$V$ is 
isomorphic
to $W$ or $W^*$ or $V|_H$ is reducible.
\end{thm}

\begin{proof} Up to Lemma~\ref{lem:5cases}, we did
not require the $kH$-module $W$ to be irreducible. Moreover, 
Proposition~\ref{lin-la-gen} and Lemma~\ref{lem:5cases} show that if $W$ is a $Q$-linear large $kH$-module,
then $V|_H$ is reducible unless $e(\lambda)\leq 2$. If $e(\lambda) = 2$,
then $V$ is the heart of the symmetric or alternating square of the natural
$kG$-module $W$. But then it is clear that if $W$ is a reducible $kH$-module,
the heart of the alternating or symmetric square is again reducible
as a $kH$-module. Hence, the irreducibility of $V|_H$ implies that $W|_H$
is irreducible and then the result follows from the preceding two lemmas.
\end{proof}
\end{section}



\begin{section}{The nonrestricted case}
In this section we assume that $W$ is a $Q$-linear large $kH$-module, as usual, 
$G\subset \SL(W)$ is the smallest classical group containing the image
of $H$ under the corresponding representation, and 
$V$ is an irreducible $kG$-module with highest weight $\lambda$. Moreover,
we assume that $V$ is not a Frobenius twist of $W$ or $W^*$.
Having treated the case 
where $\lambda$ is restricted we now turn to the case where the highest weight 
$\lambda$ is not restricted.
This hypothesis, together with Steinberg's tensor product theorem, implies that 
$V$ decomposes as a twisted 
tensor product. Write 
$\lambda = \sum_{i=0}^t ({\rm char}(k))^i\mu_i$, where the 
$\mu_i$, for $0\leq i\leq t$, are 
restricted highest weights, and let $F$ denote the standard Frobenius 
morphism of $G \subset {\rm GL}(W)$. If exactly one of the $\mu_i$ is 
nonzero, then
$H$ acts irreducibly on $V$ if and only if $H$ 
 acts irreducibly on the irreducible $kG$-module with highest weight $\mu_i$.
Since there are no such configurations, by Theorem~\ref{thm_restricted},
 we assume that at least $2$ of the 
$\mu_i$ are nonzero.

\begin{lemma}\label{NR_e=1}
If $\lambda = \sum_{i=0}^t ({\rm char}(k))^i\mu_i$, and $e(\mu_i)> 1$ for 
some $i$, 
then $V|_H$ is reducible.
\end{lemma}

\begin{proof} Our hypothesis on $\lambda$ implies that 
$V = V(\mu_0)\otimes V(\mu_1)^F \otimes \dots \otimes V(\mu_t)^{F^t}$, where 
$V(\mu_i)$ is a restricted 
highest weight module for $G$. If $e(\mu_i)> 1$ then Theorem~\ref{thm_restricted} 
implies that $V(\mu_i)|_H$ is reducible. As the tensor product is distributive 
our claim follows.\end{proof}

So without loss of generality, we now assume that each nontrivial tensor 
factor of $V$ is a 
Frobenius twist of $W$ or $W^\ast$.

\begin{prop}\label{NR_linear}
Assume $H$ is of type $\Lin$, 
and $V$ is an irreducible $kG$-module with highest weight 
$\lambda = \sum_{i=0}^t ({\rm char}(k))^i\mu_i$. If $\mu_i\ne0\ne\mu_j$ for
some $0\leq i\ne j\leq t$, then $V|_H$ is reducible.
\end{prop}

\begin{proof} Suppose the contrary. Then
Lemma~\ref{NR_e=1} implies that $V$ is a tensor product of $e$ 
Frobenius twists of restricted $kG$-modules, each of which is isomorphic
to either $W$ or $W^\ast$, where $2 \leq e \leq t+1$. 
If $H = \Lin_{r+1}(q)$ with $q \geq 4$ and $r \geq 1$, then 
\cite[3.2]{MaTi} implies that $V|_H$ is reducible. 
If $H= \Lin_{r+1}(3)$, with $r\geq 2$, then
\cite[3.4]{MaTi} implies that the tensor product of two  
$kH$-modules is reducible unless one 
of the factors is a Weil module. 
It follows from the basic properties of Weil modules (see \cite{GT_low}) that if $W$ is $Q$-linear large then
 none of $W$, $W^*$, or any Frobenius twists of these modules
are Weil modules. 

Hence, by Lemma~\ref{lem:basecases} and the above considerations, we are 
reduced to the case where $H = \Lin_{r+1}(2)$ 
with $r \geq 4$. 
We now refer to \cite[3.3]{MaTi} to see that 
 the tensor product of any two non-trivial $kH$-modules is reducible,
thus establishing the proposition.\end{proof} 

Let $P< H$ be as in Sections~\ref{sec:intro} and \ref{sec:prelim},
the stabilizer of a singular $1$-space of the natural module for $H$.
Write as usual $P = QL$, where $Q=O_p(P)$. Using Table~\ref{tab:HP},
we see that
 $L$ has up to three orbits ${\mathcal O}_i$, $1 \leq i  \leq 3$ 
on the non-trivial linear characters of $Q$. Define 
$W_{{\mathcal O}_i} := \oplus_{\lambda \in {\mathcal O}_i} W_{\lambda}$
and note that $W_{{\mathcal O}_i}$ is a $P$-module. For a $k$-subspace $U$ 
of $W$,
we write $U^F$ for the image of $U$ under the semilinear map induced by 
$F$.

\begin{lemma}\label{lem:equiv} 
With the notation as above we have that  $W_{{\mathcal O}_i}$ and 
$(W_{{\mathcal O}_i})^F$ are equivalent $kQ$-modules
for each $i$. 
\end{lemma}

\begin{proof}  The $kQ$-module
$W_{{\mathcal O}_i}$ is a direct sum of $|{\mathcal O}_i|$ isotypic components. 
As $F$ is semilinear we see that 
$W_{{\mathcal O}_i}^F$ is also a direct sum of $|{\mathcal O}_i|$ isotypic 
components. Now, using Table~\ref{tab:HP}, we observe that 
the $P$-orbits on ${\rm Hom}(Q/[Q,Q],k^*)$ have pairwise distinct lengths, 
and thus 
$ W_{{\mathcal O}_i}$ and $ (W_{{\mathcal O}_i})^F$ are equivalent $kQ$-modules. 
\end{proof}

The proof of the lemma shows that for $\chi\in{\rm Hom}(Q/[Q,Q],k^*)$, with
$W_\chi \neq 0$, $F$ can be modified (multiplied) by an 
element of $P$ to guarantee that $W_\chi$ is $F$-invariant. Henceforth we
assume this is the case. (Of course the
 element of $P$ needed to 
guarantee this is not independent of our choice of $\chi$.)

\begin{lemma}\label{lem:F-inv}
If $\chi$ is a linear character of $Q$ and 
$S \subset W_\chi$ is an $L_\chi$-invariant $k$-subspace, 
then we can choose $F$ such that $W_\chi = (W^F)_\chi$ and 
$S^F$ is an $F(L_\chi)$-invariant $k$-subspace. 
\end{lemma}
\begin{proof} We choose $F$ as discussed above so that $W_\chi$ is 
$F$-invariant. The second assertion follows from the fact that
semilinear maps preserve isomorphisms and centralizers; 
that is, $F(C_{\GL(W_\chi)}(L_\chi)) = C_{\GL(W_\chi)}(F(L_\chi))$ and hence 
the endomorphism algebras of $W_\chi|_{L_\chi}$ and of $W_\chi|_{F(L_\chi)}$
 are mapped 
isomorphically one onto the other and in particular the fields of definition 
are preserved.
\end{proof}

\begin{lemma}\label{NR_atmost2}
Assume $W$ is an irreducible $Q$-linear large module for $H$. 
If  $H$ is not of type $\Lin$, $\lambda = \sum_{i=0}^t ({\rm char}(k))^i\mu_i$, 
and $V|_H$ is irreducible, then the 
number of non-trivial tensor factors is at most two, and moreover, 
the $kH$-modules $W$ 
and $W^F$ are inequivalent.
\end{lemma}
 
\begin{proof} As in the proof of Proposition~\ref{NR_linear},
 we may assume that 
$V$ is a tensor product of $e$ Frobenius twists of restricted irreducible
$kG$-modules each of which is isomorphic to one of
$W$ and $W^\ast$, and where $2 \leq e \leq t+1$. 
Now $F$ acts semilinearly on $W$
and is a bijective morphism of the group $G$. Hence, 
$F^i(H)$ is a subgroup of $G$ acting irreducibly on 
$W^{F^i}$
 and on $(W^{\ast})^{F^i}$, for all $i\geq 0$. 
If $W$ and $W^F$ are equivalent $kH$-modules,
 then any tensor product involving Frobenius 
twists of $W$ and $W^{\ast}$
must necessarily be reducible; hence 
$W$ and $W^F$ are inequivalent $kH$-modules.

We now argue as in the proofs of Lemma~\ref{newupper} and 
Proposition \ref{lin-la-gen} to produce a subspace of 
$V$ which is invariant under $P_\chi = QL_\chi$ for some linear character 
$\chi \in Q^*$ such that $W_\chi \neq 0$.
As usual we let $S \subset W_\chi$ be an irreducible $P_\chi$-submodule. 
Then the $P_\chi$-module $W^F$ has a 
submodule $S^F$ of dimension equal to $\dim S$. By Lemma~\ref{lem:F-inv},
all 
tensor factors of $V$ possess
a $P_\chi$-invariant submodule of dimension equal to $\dim S$. This 
yields an upper bound for ${\rm dim}(V)$ namely 
$$ \dim V \leq (\dim S)^e[H:P_\chi] = (\dim S)^e[H:P][L:L_\chi].$$ 
 
Clearly, $\dim W \geq (\dim S)[P:P_\chi] = (\dim S)[L:L_\chi]$, 
which yields a lower bound for 
$\dim V$, and we therefore have the following two inequalities:
$$ (\dim S)^e[L:L_\chi]^e \leq \dim V 
\leq (\dim S)^e[H:P][L:L_\chi].$$

In the proof of Proposition~\ref{lin-la-gen} we checked that 
$${[L:L_\chi]-1\choose 3} > [H:P_\chi]$$ for all $H \neq  
{\mathbf O}_7(2), \Om_8^+(2),{\mathbf O}_9(2)$. As 
${[L:L_\chi]-1\choose 3} < [L:L_\chi]^3$, we have that $e \leq 2$ for all 
$H \neq {\mathbf O}_7(2), \Om_8^+(2),{\mathbf O}_9(2)$.

If $H = {\mathbf O}_7(2)$, then arguing as in the proof of Lemma~\ref{lem:5cases},
we see that $\dim V\leq 720$. Since $\sqrt[3]{720}<9$ and any $Q$-linear large module
for $H$ has dimension at least 14, $e$ can be at most 2. For $H = \Om_8^+(2)$, we have $\dim V\leq 6075$, 
while $\dim W\geq 104$, and again we have $e\leq 2$. Finally, if $H={\mathbf O}_9(2)$, then we check that 
$\dim W\geq 28 \cdot 4 = 112$ and thus 
$(\dim W)^3$ exceeds the degree of 
any irreducible $kH$-representation, completing the proof that $e \leq 2$.\end{proof}

\begin{prop} If $H$ is not of type $\Lin$ and $V = W \otimes W^{F^i}$
or  $V = W \otimes (W^\ast)^{F^i}$, then $V|_H$ is reducible unless:
\begin{enumerate}
\item $q \leq 3$, or
\item $q$ is even and $H={\mathbf O}_{2r+1}(q)$.
\end{enumerate}
\end{prop}
	
\begin{proof} The content of Theorem 1.1 of \cite{MaTi} is that 
$H = X_r(q)$ acts reducibly on tensor 
products of irreducible cross-characteristic $H$-modules unless 
one of the following holds:
\begin{itemize}
\item $q \leq 3$;
\item $H = \O_{2r+1}(q)$ with $q$-even;
 \item $H = \Sp_{2r}(5)$; or
\item $H=\Lin_3(4)$.
\end{itemize}
The last case is ruled out by Lemma~\ref{lem:basecases}.
 The case $H= \Sp_{2r}(5)$ can be 
eliminated using Proposition 5.2 (ii) of \cite{MaTi}, as our hypothesis that $W$ is $Q$-linear large
implies that $W$ is not a Weil module.\end{proof}

To complete the proof of Theorem 1, we must handle case 2. of the previous 
result in case $q\geq 4$. This is covered by the following:

\begin{prop}\label{prop:final} If $q \geq 4$ is even, $H={\mathbf O}_{2r+1}(q)$ with 
$r\geq 3$, 
$W$ is an absolutely irreducible $kH$-module, and 
$V = W \otimes W^{F^i}$
or  $V = W \otimes (W^\ast)^{F^i}$, then $V|_H$ is reducible.
\end{prop}

Before beginning the proof of the proposition, we set up some further notation
and establish 2 lemmas.

Let $P_2$ denote the stabilizer in $H$ of a singular $2$-space. We denote the 
unipotent radical of $P_2$ 
by $Q_2$ and fix a 
Levi factor $L_2$, so $L_2\cong \GL_2(q){\rm O}_{2r-3}(q)$. Now 
$Z(Q_2)$ is an elementary abelian group of order $q^3$,
the natural module of the $O_3(q)\cong {\rm PGL}_2(q)$ subgroup of $L_2$.
 By $W_\lambda$ we denote 
the $Z(Q_2)$ weight 
spaces of $W$, for $\lambda\in {\rm Hom}(Z(Q_2),k^*)$; hereafter, write
$Z(Q_2)^*$ for this last ${\rm Hom}$ space. The group $L_2$ has three orbits on 
the non-trivial characters of 
$Z(Q_2)$ and the results of Table~\ref{tab:HP} show that at least 
$q(q-1)^2/2$ of the weight spaces are 
non-trivial. 

As the $L_2$-orbits on $Z(Q_2)^*$ have distinct lengths, we can
argue as in the proof of Lemma~\ref{lem:equiv} to obtain:

\begin{lemma}\label{lem:equiv2}
Denote by ${\mathcal O}_i$, for $i\in I$, the orbits of $L_2$ on 
$Z(Q_2)^*$. We have that
$W_{{\mathcal O}_i}$ and $(W_{{\mathcal O}_i})^F$ are equivalent $kQ_2$-modules
for all $i$. 
\end{lemma}

Now we require some information about the irreducible characters of 
$Q_2$ and the construction of those which are nonlinear. 

\begin{lemma} \label{Q2} The irreducible characters of $Q_2$ are either
\begin{itemize}
\item linear, and there are $|Q_2|/q$ of them, or
\item of degree $[Q_2:Z(Q_2)]^{1/2}$, and there are $q^3-q^2$ of them. 
\end{itemize}

\end{lemma}

\begin{proof} The commutator subgroup of $Q_2$ has order $q$ so the 
statement concerning linear characters is clear. 
Let $R_2$ be the subgroup of $Q_2$ generated by the long root subgroups 
(with respect to a fixed torus inside $P_2$). Note that 
$[R_2,R_2] = [Q_2,Q_2]$ and hence $Z(Q_2)$ has exactly $q^3-q^2$ characters
 which do not have $[R_2,R_2]$ in their kernel. 
Moreover, we note that the group $R_2$ is special, and such that 
$[x,R_2] = Z(R_2)$ for all $x \in R_2 \setminus Z(R_2)$.
Thus all nonlinear irreducible characters of $R_2$ have degree 
$[R_2:Z(R_2)]^{1/2} = [Q_2:Z(Q_2)]^{1/2}$.

Now let $\lambda$ be a character of $Z(Q_2)$ which does not have 
$[R_2,R_2]$ in its kernel. Let $A$ be an 
abelian subgroup of $R_2$ such that $\langle A, Z(R_2) \rangle$ is a 
maximal abelian subgroup of $R_2$. Since $A\cap Z(R_2) = 1$, and 
$Z(Q_2)\cap R_2 = Z(R_2)$, we have that $A\cap Z(Q_2) = 1$ as well.
As $\langle A, Z(Q_2) \rangle$ is an abelian subgroup of $Q_2$ we see 
that $\lambda$ extends to a linear character 
${\tilde \lambda}$ of $\langle A, Z(Q_2) \rangle$. Thus the degree of 
${\tilde \lambda}\uparrow^{Q_2}$ is 
$[Q_2:Z(Q_2)]^{1/2}$. The restriction of ${\tilde \lambda}\uparrow^{Q_2}$ to 
$Z(R_2)$ is nontrivial and again since $R_2$ is special,
$ {\tilde \lambda}\uparrow^{Q_2}\downarrow_{R_2}$ is irreducible; hence  
${\tilde \lambda}\uparrow^{Q_2}$ is an 
irreducible character of degree claimed above. Now the construction of
 ${\tilde \lambda}\uparrow^{Q_2}$ can be 
carried out for all characters of $Z(Q_2)$ which do not have $[R_2,R_2]$ in 
their kernel. Evidently all such characters are pairwise distinct as their 
restrictions to $Z(Q_2)$ are distinct as well. 
So summing the squares of the degrees of the irreducible $Q_2$ characters 
constructed so far yields $|Q_2|$. Thus 
the first orthogonality relation of characters shows that we have 
constructed all irreducible characters of $Q_2$.\end{proof}

\begin{proof}[Proof of Proposition~\ref{prop:final}] Arguing as in the proof of 
Lemma~\ref{lem:F-inv} and using
Lemma~\ref{lem:equiv2}, we 
have that $W_\lambda$ and 
$(W^F)_\lambda$ are equivalent 
as $Q_2$-modules. We note here that 
$\dim W_\lambda = m_\lambda [Q_2:Z(Q_2)]^{1/2}$, where $m_\lambda$
is the multiplicity of the  $Q_2$ character ${\tilde \lambda}\uparrow^{Q_2}$
in $W_\lambda$. 

This leads to a lower bound
$$ \dim V\geq (m_\lambda)^2[Q_2:Z(Q_2)](q^2-q)^2(q-1)^2/4.$$

We now derive an upper bound for $\dim V$. Recall from above 
that the subgroup $R_2$ of $Q_2$ generated by long root elements is
special and that any irreducible representation of $Q_2$ which restricts 
nontrivially to $Z(R_2)$ restricts  
irreducibly to $R_2$. As $q$ is even, if
$\lambda \in Z(Q_2)^*$ and  $Z(R_2) \not \subset {\rm Ker}(\lambda)$, then
$Z(Q_2)$ acts as minus the identity on $W_\lambda$ and thus acts 
trivially on $W_\lambda \otimes W_\lambda^F$. 
Thus the action of $Q_2$ on 
$W_\lambda \otimes W_\lambda^F$ is $m_\lambda^2$ times the regular character of 
$Q_2/Z(Q_2)$.
In particular, 
$C_{W_\lambda \otimes W_\lambda^F}(Q_2) \neq 0$ and this subspace contains a 
$Q_2{\rm O}_{2r-3}(q)$-submodule 
$S_\lambda$ of dimension at most $(m_\lambda)^2$. 
Thus we see that the $P_2$-module  
$\sum_\lambda C_{W_\lambda \otimes W_\lambda^F}(Q_2)$, on which $Q_2$ acts
trivially,  
affords a $\GL_2(q) \times {\rm O}_{2r-3}(q)$ module. As 
the ${\rm O}_{2r-3}(q)$-submodule 
$S_\lambda$ is contained in  $\sum_\lambda C_{W_\lambda \otimes W_\lambda^F}(Q_2)$, 
we see 
by Frobenius reciprocity that 
$\sum_\lambda C_{W_\lambda \otimes W_\lambda^F}(Q_2)$ contains a quotient of 
the $P_2$-module 
$S_\lambda\uparrow_{Q_2:{\rm O}_{2r-3}(q)}^{P_2}$, whose
socle contains a nontrivial irreducible ${\rm GL}_2(q)\times {\rm O}_{2r-3}$
submodule 
$S$. As 
$\dim S \leq (q+1)m_\lambda^2$, we see that 
$$\dim V \leq (q+1)m_\lambda^2 [H:P_2] =  
(q+1)m_\lambda^2 (q^{2r}-1)(q^{2r-2}-1)/(q^2-1)(q-1)$$
$$= m_\lambda^2 (q^{2r}-1)(q^{2r-2}-1)/(q-1)^2. $$



Combining the upper and lower bounds and using the fact that 
$[Q_2:Z(Q_2)] = q^{2(2r-4)}$, leads to the inequality
$$q^{2(2r-4)}(q^2-q)^2(q-1)^2(q-1)^2 < 4(q^{2r}-1)(q^{2r-2}-1),$$ which 
reduces to 
$$q^{(4r-6)}(q-1)^6 < 4(q^{2r}-1)(q^{2r-2}-1),$$ and leads to a 
contradiction when 
$q \geq 8$ and $r\geq 3$. 

When $q = 4$ we improve our upper bound for $\dim V$ by $1/2$ using 
induction. 
All the modular character tables for ${\rm O}_3(4) \cong {\rm SL}_2(4)$ and 
${\rm O}_5(4) \cong {\rm Sp}_4(4)$ are in 
\cite{GAP} and we see that $W \otimes W'$ is reducible, for any pair of irreducible modules $W$ and $W'$,
 and so in particular,
$W \otimes W^F$ is reducible. Thus, by dualizing if necessary, we see that 
$W_\lambda \otimes (W^F)_\lambda$ has an ${\rm O}_{2r-3}(4)$-submodule of 
dimension at most 
$m_\lambda^2/2$ (rather than $m_\lambda^2$ as before), which leads to the 
improved upper 
bound. The improved inequality yields a contradiction when $q=4$.

This completes the proof of the proposition.\end{proof}
 
We have now established:

\begin{thm} Let $W$ be an irreducible $Q$-linear large $kH$-module, 
$G\subset \GL(W)$ the smallest classical group
containing the image of $H$, and $V$ a tensor decomposable non-restricted 
irreducible $kG$-module. Let $F$
be a fixed ${\rm char}(k)$-power Frobenius morphism of the group $G$.
If $V$ is a twisted tensor product of restricted irreducible $kG$-modules, 
then 
$V|_H$ is reducible unless $q \leq 3$ and 
$V$ is a Frobenius twist of a module of the form $X \otimes Y^F$, where 
$X,Y\in\{W,W^*\}$, and $Y^F$ is a Frobenius twist of $Y$
such that $X|_H$ and $Y^F|_H$ are inequivalent $kH$-modules. 
\end{thm}

\end{section}

\bibliographystyle{alpha}

\end{document}